\documentclass[a4paper,11pt]{article} %,draft
\usepackage{amsmath,amssymb,enumerate,color}
\usepackage{amsthm,color}
\usepackage[colorlinks=true,linkcolor=magenta,citecolor=blue]{hyperref}
\usepackage{txfonts}
\usepackage{bigints}
\usepackage{comment}
\usepackage{graphicx}
\usepackage{enumitem}

\usepackage{bm}

\newcommand{\R}{\mathbb{R}}
\newcommand{\C}{\mathbb{C}}

\newcommand{\N}{\mathbb{N}}
\newcommand{\ep}{\varepsilon}

\newcommand{\ity}{\infty}
\newcommand{\dps}{\displaystyle}
\newcommand{\f}{\displaystyle\frac}

\newtheorem{theorem}{Theorem}[section]
\newtheorem{lemma}[theorem]{Lemma}

\theoremstyle{remark}
\newtheorem{remark}{Remark}[section]
\theoremstyle{definition}

\newtheorem{definition}{Definition}[section]

\numberwithin{equation}{section}

\makeatletter
\def\@cite#1#2{[{{\bfseries #1}\if@tempswa , #2\fi}]}
\makeatother                  
\begin{document}
%\linenumbers 
\begin{center}
\Large{{\bf
Blow-up and lifespan estimates for solutions to the weakly coupled system of nonlinear damped wave equations outside a ball}}
\end{center}

\vspace{5pt}

\begin{center}
Tuan Anh Dao
\footnote{
School of Applied Mathematics and Informatics, Hanoi University of Science and Technology, No.1 Dai Co Viet road, Hanoi, Vietnam/Institute of Mathematics, Vietnam Academy of Science and Technology, No.18 Hoang Quoc Viet road, Hanoi, Vietnam.
E-mail:\ {\tt anh.daotuan@hust.edu.vn} (\textit{Corresponding author})},
Masahiro Ikeda
\footnote{
Department of Mathematics, Faculty of Science and Technology, Keio University, 3-14-1 Hiyoshi, Kohoku-ku, Yokohama, 223-8522, Japan/Center for Advanced Intelligence Project, RIKEN, Japan.
E-mail:\ {\tt masahiro.ikeda@keio.jp/masahiro.ikeda@riken.jp}}\ 

\end{center}

%%%%     ABSTRACT     %%%%%%
\newenvironment{summary}{\vspace{.5\baselineskip}\begin{list}{}{%
     \setlength{\baselineskip}{0.85\baselineskip}
     \setlength{\topsep}{0pt}
     \setlength{\leftmargin}{12mm}
     \setlength{\rightmargin}{12mm}
     \setlength{\listparindent}{0mm}
     \setlength{\itemindent}{\listparindent}
     \setlength{\parsep}{0pt}
     \item\relax}}{\end{list}\vspace{.5\baselineskip}}
\begin{summary}
{\footnotesize {\bf Abstract.}
In this paper, we consider the initial-boundary value problems with several fundamental boundary conditions (the Dirichlet/Neumann/Robin boundary condition) for the multi-component system of semi-linear classical damped wave equations outside a ball. By applying a test function approach with a judicious choice of test functions, which approximates the harmonic functions being subject to these boundary conditions on $\partial \Omega$, simultaneously we have succeeded in proving the blow-up result in a finite time as well as in catching the sharp upper bound of lifespan estimates for small solutions in two and higher spatial dimensions. Moreover, such kind of these results will be discussed in one-dimensional case at the end of this work.
}
\end{summary}

\noindent{\footnotesize{\it Mathematics Subject Classification}\/ (2020): %
%Primary:
35B44; %	Blow-up in context of PDEs
35A01; %	Existence problems: global existence, local existence, non-existence
35L15;%	Initial value problems for second-order hyperbolic equations
35L05%	Wave equation
}
\tableofcontents

\section{Introduction}
This paper is mainly concerned with investigating upper bound of lifespan estimates for small solutions to the following weakly coupled system of semi-linear classical damped wave equations in an exterior domain:
\begin{equation}\label{Equation_Main}
\begin{cases}
\partial^2_t u_1(t,x)- \Delta u_1(t,x)+ \partial_t u_1(t,x)= |u_{\mathtt{k}}(t,x)|^{p_1}, &(t,x) \in (0,T) \times \Omega, \\
\partial^2_t u_2(t,x)- \Delta u_2(t,x)+ \partial_t u_2(t,x)= |u_1(t,x)|^{p_2}, &(t,x) \in (0,T) \times \Omega, \\
\quad \vdots \\
\partial^2_t u_{\mathtt{k}}(t,x)- \Delta u_{\mathtt{k}}(t,x)+ \partial_t u_{\mathtt{k}}(t,x)= |u_{\mathtt{k}-1}(t,x)|^{p_{\mathtt{k}}}, &(t,x) \in (0,T) \times \Omega, \\
\alpha\f{\partial u_\ell}{\partial n^+}(t,x)+ \beta u_\ell(t,x)=0, &(t,x) \in (0,T) \times \partial\Omega,\,\ell=1,2,\cdots,\mathtt{k}, \\
u_\ell(0,x)= \varepsilon u_{0,\ell}(x),\quad \partial_t u_\ell(0,x)= \varepsilon u_{1,\ell}(x), &x \in \Omega, \,\ell=1,2,\cdots,\mathtt{k},
\end{cases}
\end{equation}
where $\mathtt{k}\ge 2$, $p_\ell>1$ with $\ell=1,2,\cdots,\mathtt{k}$ and $T>0$. The domain $\Omega$ is given by $\Omega:= \big\{x\in \R^d \,\,:\,\, |x|>1\big\}$ with $d\ge 2$, and $u_{\ell}:(0,T)\times\Omega\rightarrow \C$ with $\ell=1,2,\cdots,\mathtt{k}$ denote unknown functions to the problem (\ref{Equation_Main}). In addition, $n^+$ stands for the outward unit normal on $\partial\Omega$. The positive constant $\varepsilon$ presents the size of initial data. The given functions $u_{0,\ell}$ and $u_{1,\ell}$ with $\ell=1,2,\cdots,\mathtt{k}$ represent the shape of the initial data. Let $\alpha,\beta\in \mathbb{R}$ be real constants satisfying $(\alpha,\beta)\ne (0,0)$. The boundary condition is called the Dirichlet boundary if $\alpha=0$, the Neumann boundary condition if $\beta=0$ and the Robin boundary condition otherwise. Here $\Delta=\Delta_{\alpha,\beta}$ denotes the Laplace operator dependent on the boundary condition in the open subset $\Omega$ of $\R^d$ (see Notations in the end of this section for the precise definition). \medskip

To get started, let us make some attention to the Cauchy problem of \eqref{Equation_Main} in the whole space, i.e. $\Omega=\R^d$. Concerning $\mathtt{k}=1$, the single classical semi-linear damped wave equation
\begin{equation}\label{Equation_1}
\begin{cases}
\partial_t^2u-\Delta u+\partial_tu=|u|^p, & (t,x)\in (0,T) \times \R^d,\\
u(0,x)= \varepsilon u_0(x),\quad \partial_t u(0,x)= \varepsilon u_1(x), &x\in \R^d,
\end{cases}
\end{equation}
Todorova-Yordanov \cite{TodorovaYordanov2001} introduced the so-called Fujita exponent $p_{\rm Fuj}(d):= 1+ \frac{2}{d}$ (see more \cite{Fujita1966,Hayakawa1973,Weissler1981}, the well-known Fujita exponent for the semi-linear heat equation), which classifies between the global (in time) existence of small solutions to \eqref{Equation_1} for $p>p_{\rm Fuj}(d)$ and a small data blow-up result in the inverse case $1<p<p_{\rm Fuj}(d)$. Especially, the treatment of the critical case $p=p_{\rm Fuj}(d)$, verified by Zhang \cite{Zhang2001} later, is also to conclude nonexistence of global solutions to \eqref{Equation_1} even for small data. When the blow-up phenomenon in finite time occurs, the maximal existence time of solutions to \eqref{Equation_1}, which is the so-called lifespan, can be estimated by a series of previous works \cite{LiZhou1995, KiraneQafsaoui2002,Nishihara2003, Nishihara2011,IkedaOgawa2016,LaiZhou2019,FujiwaraIkedaWakasugi2019,IkedaWakasugi2020} and references therein as follows:
$$ {\rm LifeSpan}(u)\sim
\begin{cases}
\ep^{-\frac{2(p-1)}{2-d(p-1)}} &\text{ if }\ \ 1<p<p_{\rm Fuj}(d), \\
\exp\big(C\ep^{-(p-1)}\big) &\text{ if }\ \ p=p_{\rm Fuj}(d).
\end{cases}  $$
Regarding $\mathtt{k}=2$ of \eqref{Equation_Main} in $\R^d$ the authors in \cite{SunWang2007,Narazaki2009,NishiharaWakasugi2014,DaoReissig2021} described the following critical curve in the $p_1-p_2$ plane:
$$ \gamma_{\max}(p_1,p_2):=\frac{\max\{p_1,p_2\}+1}{p_1p_2-1}=\frac{d}{2}. $$
More specifically, they proved that the global (in time) small data Sobolev solutions exist if $\gamma_{\max}(p_1,p_2)<d/2$, meanwhile, every non-trivial local (in time) weak solution blows up in finite time if $\gamma_{\max}(p_1,p_2)\geqslant d/2$. One point worth noticing in \cite{NishiharaWakasugi2014,NishiharaWakasugi2015} is that the sharp upper bound estimate for the lifespan of solutions in the subcritical case $\gamma_{\max}(p_1,p_2)>d/2$ was given by
$$ {\rm LifeSpan}(u_1,u_2) \le C\varepsilon^{-\frac{1}{\gamma_{\max}(p_1,p_2)-d/2}}. $$
Quite recently, for the purpose of making the study of lifespan self-contained, Chen-Dao \cite{ChenDao2021} have found out the sharp lifespan estimates in the critical case $\gamma_{\max}(p_1,p_2)=d/2$, namely,
$$ {\rm LifeSpan}(u_1,u_2)\sim
\begin{cases} 
\exp\left( C\varepsilon^{-(p_1-1)} \right)&\mbox{if}\ \ p_1=p_2,\\
\exp\left(C\ep^{-(p_1 p_2-p_{\mathrm{Fuj}}(d))}\right) &\mbox{if}\ \ p_1\neq p_2.
\end{cases} $$
To demonstrate this, the authors have applied a suitable test function method linked to the technical estimates for nonlinear differential inequalities and have constructed polynomial-logarithmic type time-weighted Sobolev spaces as well in terms of deriving upper bound estimates and lower bound estimates for the lifespan, respectively. For the more general cases $\mathtt{k}\ge 3$ of \eqref{Equation_Main} in $\R^d$ , Takeda \cite{Takeda2009} obtained both the global existence and the finite time blow-up result for small solutions to establish the critical condition
$$ \gamma_{\max}:= \max\{\gamma_1,\gamma_2,\cdots,\gamma_{\mathtt{k}}\}= \frac{d}{2} $$
for any $d\le 3$, which was further extended by Narazaki \cite{Narazaki2011} for any $d\ge4$ thanks to using weighted Sobolev spaces. Here we note that the aforementioned parameters $\gamma_\ell$ with $\ell=1,2,\cdots,\mathtt{k}$ are introduced as in \eqref{Parameter_gamma}. Not much later, Nishihara-Wakasugi \cite{NishiharaWakasugi2015} improved these results for any $d\ge 1$ by the application of a weighted energy method. One may see that the authors in the latter paper also showed the following lifespan estimates for solutions from both the above and the below:
\begin{align*}
c\varepsilon^{-\frac{1}{\gamma_{\max}-d/2}+\delta}\le {\rm LifeSpan}(u_1,u_2,\cdots,u_{\mathtt{k}})\le C\varepsilon^{-\frac{1}{\gamma_{\max}-d/2}}
\end{align*}
for any small number $\delta>0$. In other words, they gave an almost optimal estimate for the lifespan but it seems to be far from lower bound one to upper bound one (see also \cite{HayashiNaumkinTominaga2015}).\medskip

Turning back to our models \eqref{Equation_Main}, as far as there have been a lot of investigations in the study of the exterior problems for semi-linear classical damped wave equation. It is significant to recognize that the essential difference to the initial problem originates from the influence of reflection at the boundary and the lack of symmetric properties including scale-invariance, rotation-invariance and so on. Let us recall several previous literatures involving the the initial-boundary value problem for the single equation in an exterior domain, i.e. \eqref{Equation_Main} with $\mathtt{k}=1$, as follows:
\begin{equation}\label{Equation_2}
\begin{cases}
\partial^2_t u(t,x)- \Delta u(t,x)+ \partial_t u(t,x)= |u(t,x)|^{p}, &(t,x) \in (0,T) \times \Omega, \\
u(t,x)=0, &(t,x) \in (0,T) \times \partial\Omega, \\
u(0,x)= \varepsilon u_0(x),\quad \partial_t u(0,x)= \varepsilon u_1(x), &x \in \Omega.
\end{cases}
\end{equation}
Particularly, Ikehata \cite{Ikehata2002,Ikehata2004,Ikehata2005} and Ono \cite{Ono2003} succeeded in proving the existence of global solutions to \eqref{Equation_2} if $p>p_{\rm Fuj}(2)$ for $d=2$ and $1+\frac{4}{d+2}<p\le 1+\frac{2}{d-2}$ for $d=3,4,5$. Meanwhile, one can find in the paper of Ogawa-Takeda \cite{OgawaTakeda2009} that the solutions to \eqref{Equation_2} cannot exist globally if $1<p<p_{\rm Fuj}(d)$ for any $d\ge 1$ by using Kaplan-Fujita method. Afterwards, the critical case $p=p_{\rm Fuj}(d)$ was independently filled by Fino-Ibrahim-Wehbe \cite{FinoIbrahimWehbe2017} and Lai-Yin \cite{LaiYin2017}, where the finite time blow-up phenomena also occurs. To demonstrate these blow-up results, the authors employed a suitable test function method essentially, which was originally developed by Baras-Pierre \cite{BarasPierre1985} (see also \cite{MitidieriPohozaev2001,Zhang2001}). More precisely, the aid of the first eigenfunction of $-\Delta$ as well as the corresponding first eigenvalue over $\Omega$ was taken into consideration in \cite{OgawaTakeda2009,FinoIbrahimWehbe2017}, whereas the employment of some properties on the Riemann-Liouville fractional derivative and the harmonic function in $\Omega$ came into play in \cite{LaiYin2017}. However, their approaches seem to be difficult to apply directly in catching the lifespan estimates for solutions to \eqref{Equation_2} due to technical issues. More recently, Sobajima \cite{Sobajima2019-1} (see also Ikeda-Sobajima \cite{IkedaSobajima2019-1}) established the following upper bound estimates for lifespan of solutions to \eqref{Equation_2}:
\begin{equation}\label{Lifespan_Single-Eq}
{\rm LifeSpan}(u)\lesssim
\begin{cases}
\ep^{-\frac{p-1}{2-p}}\Big(\log \big(\ep^{-1}\big)\Big)^{\frac{p-1}{2-p}} &\text{ if }\ \ 1<p<2 \text{ and }d=2, \\
\exp\exp\Big(C\ep^{-1}\Big) &\text{ if }\ \ p=2 \text{ and }d=2, \\
\ep^{-\frac{2(p-1)}{2-d(p-1)}} &\text{ if }\ \ 1<p<p_{\rm Fuj}(d) \text{ and }d\ge 3, \\
\exp\Big(C\ep^{-(p-1)}\Big) &\text{ if }\ \ p=p_{\rm Fuj}(d) \text{ and }d\ge 3. \\
\end{cases}
\end{equation}
At first sight, these estimates are exactly the same as those for the corresponding Cauchy problem \eqref{Equation_1} in higher dimensions $d\ge 3$ but the difference comes from the case of $d=2$, especially, the double exponential type appears in the critical exponent $d=2$. This different behavior can be interpreted as the influence of recurrency of the Brownian motion in two-dimensional case. Regarding the Robin initial-boundary value problem for the single equation in an exterior domain, i.e. \eqref{Equation_Main} with $\mathtt{k}=1$, very recently Ikeda-Jleli-Samet \cite{IkedaJleliSamet2020} have investigated both the existence and nonexistence of gloabl solutions to the following problem:
\begin{equation}\label{Equation_3}
\begin{cases}
\partial^2_t u(t,x)- \Delta u(t,x)+ \partial_t u(t,x)= |u(t,x)|^{p}, &(t,x) \in (0,T) \times \Omega, \\
\f{\partial u_\ell}{\partial n^+}(t,x)+\beta u_\ell(t,x)=f(x), &(t,x) \in (0,T) \times \partial\Omega, \\
u(0,x)= u_0(x),\quad \partial_t u(0,x)= u_1(x), &x \in \Omega,
\end{cases}
\end{equation}
with a nontrivial Robin boundary condition (see also \cite{JleliSamet2019} and the references therein), i.e. $\beta>0$ and a function $f\not\equiv 0$. In their paper, one should recognize that the information about estimating lifespan of solutions has been not obtained to \eqref{Equation_3} even for a zero Robin boundary condition. Additionally, the fact is that at present there is not any result for blow-up of damped wave equation with the Neumann boundary condition, in particular, for \eqref{Equation_3} with $\beta=0$ as far as we know. \medskip

To the best of the authors' knowledge, no work in terms of the study of the lifespan estimates for solutions to \eqref{Equation_Main} exists in the literature so far even when this system consists two components, i.e. \eqref{Equation_Main} with $\mathtt{k}=2$. Motivated strongly by \cite{ChenDao2021,IkedaSobajima2019-1,IkedaJleliSamet2020}, our main goal of this paper is to indicate not only the blow-up of solutions but also the sharp upper bound of lifespan to \eqref{Equation_Main} for any $\mathtt{k}\ge 2$ in two and higher spatial dimensions. For this purpose, we appropriately determine the test function method, which approximates the harmonic functions enjoying the Dirichlet boundary condition or the Robin boundary condition on $\partial \Omega$ for $d=2$ and for any $d\ge 3$ individually, as well as effectively apply the technical derivation of lifespan estimates modified from \cite{IkedaSobajima2019-2,IkedaSobajima2019-1,ChenDao2021}. Nevertheless, as we can see later (Section \ref{Sec_Proof_Upper}), our strategy dealing with $d\ge 3$ does not work so well to explore the special case $d=2$. Hence, considering $d=2$ as an exceptional case of \eqref{Equation_Main} we will discuss the treatment of this case in the other approach. Moreover, we would say that this paper seems to be the first result to investigate the blow-up phenomenon for the system \eqref{Equation_Main}, even for the single damped wave equation, with Neumann boundary condition. For this kind of boundary condition, we want to point out that it is enough to use the unique harmonic function to treat for all spatial dimension. Finally, such kind of these results in one-dimensional case will also be remarked at the end of this work. \medskip

\textbf{The structure of this paper is organized as follows:} We state the main results including local well-posedness in the energy space, small data blow-up result and upper bound estimate for lifespan of solutions in Section \ref{Sec.Main}. Section \ref{Sec.Theorem-1} is to present the proof of local well-posedness result in the energy space. We give some of preliminary calculations of test functions in Section \ref{Test_functions} that are used in the sequel. Finally, Section \ref{Sec_Proof_Upper} is devoted to the proof of small data solution blow-up and upper bound estimate for lifespan of solutions simultaneously.\medskip

\noindent\textbf{Notations:} We give the following notations which are used throughout this paper.
\begin{itemize}[leftmargin=*]
\item We write $f\lesssim g$ when there exists a constant $C>0$ such that $f\le Cg$, and $f \sim g$ when $g\lesssim f\lesssim g$.
\item Let us denote the matrix
$$ P:=
\begin{pmatrix}
0&0&\cdots&0&p_1 \\
p_2&0&\cdots&0&0 \\
0&p_3&\cdots&0&0 \\
\vdots &\vdots& &\vdots &\vdots \\
0&0&\cdots&p_{\mathtt{k}}&0 \\
\end{pmatrix} $$
and the column vector
\begin{equation} \label{Parameter_gamma}
\gamma= (\gamma_1,\gamma_2,\cdots,\gamma_{\mathtt{k}})^{\mathtt{t}}:=(P-I_{\mathtt{k}})^{-1}(\underbrace{1,1,\cdots,1}_{\mathtt{k} \text{ times}})^{\mathtt{t}},
\end{equation}
where $I_{\mathtt{k}}$ and $(\alpha_1,\alpha_2,\cdots,\alpha_{\mathtt{k}})^{\mathtt{t}}$ stand for the identity matrix and the transposition vector of vector $(\alpha_1,\alpha_2,\cdots,\alpha_{\mathtt{k}})$, respectively. We want to point out that due to the assumption $p_\ell>1$ with $\ell=1,2,\cdots,\mathtt{k}$, it is obvious to recognize that
$$ \text{det}(P-I_{\mathtt{k}})= (-1)^{\mathtt{k}+1}\left(\prod^{\mathtt{k}}_{\ell=1} p_\ell-1\right) \neq 0. $$
So, the inverse matrix $(P-I_{\mathtt{k}})^{-1}$ exists. This means the above vector $\gamma$ is well-defined. Then, we set the element $\gamma_{\max}:= \max\{\gamma_1,\gamma_2,\cdots,\gamma_{\mathtt{k}}\}$.
\item
Let $m \in \mathbb{N}\cup\{0\}$ and $p\in [1,\infty]$. We introduce the Sobolev space $W^{m,p}(\Omega)$, which is a Banach space of measurable functions $f \colon \Omega \to \mathbb{C}$ such that $D^{\alpha}f \in L^{p}(\Omega)$ in the sense of distributions, for every multi-index $\alpha\in (\mathbb{N}\cup\{0\})^n$ with $|\alpha| \leq m$. Here $D:=(\partial_{x_1},\partial_{x_2},\cdots,\partial_{x_n})$ is a partial differential operator. The space $W^{m,p}(\Omega)$ is equipped with the norm $\|\cdot\|_{W^{m,p}(\Omega)}$ given by 
$$\|f\|_{W^{m,p}(\Omega)} := \dps\sum_{|\alpha| \leq m}\|D^{\alpha}f\|_{L^p(\Omega)}. $$
Let us denote by $W_{0}^{m,p}(\Omega)$ a closure of $\mathcal{C}_{0}^{\infty}(\Omega)$ in $W^{m,p}(\Omega)$, where $\mathcal{C}_{0}^{\infty}(\Omega)$ is the space of functions in $\mathcal{C}^{\infty}(\Omega)$ which has a compact support in $\Omega$.
\item For $m\in \mathbb{N}\cup\{0\}$, $H^{m}(\Omega)$ stands for $W^{m,2}(\Omega)$, where $H^{m}(\Omega)$ is equipped with the equivalent norm $\|\cdot\|_{H^m(\Omega)}$ given by
$$\|f\|_{H^m(\Omega)} := \left(\dps\sum_{|\alpha| \leq m}\int_{\Omega}|D^{\alpha}f(x)|^2 dx\right)^{\frac{1}{2}}. $$
Then $H^{m}(\Omega)$ is a Hilbert space endowed with the scalar product $\langle\cdot,\cdot\rangle_{H^m(\Omega)}$ given by
$$ \left<f,g\right>_{H^m(\Omega)} := \dps\sum_{|\alpha| \leq m}\Re\int_{\Omega}D^{\alpha}f(x)\overline{D^{\alpha}g(x)}dx. $$
Furthermore, $H_{0}^{m}(\Omega) := W_{0}^{m,2}(\Omega)$.

\item For $\alpha,\beta\in \mathbb{R}$ satisfying $\alpha\beta\ne 0$, we introduce a closed subspace $H^1_{\alpha,\beta}(\Omega)$ of the Sobolev space $H^1(\Omega)$ associated with the boundary condition given by
\[
   H^1_{\alpha,\beta}(\Omega):=
   \begin{cases}
   H^1_0(\Omega), &\text{if}\ \ \alpha=0,\\
   \Big\{f\in H^1(\Omega)\ :\ \partial f/\partial n^+\in L^2(\partial\Omega)\ \text{ with }\ \alpha\partial f/\partial n^++\beta f=0\Big\},&\text{if}\ \ \alpha\ne 0.
   \end{cases}
\]

\item The precise definition of the Laplace operator depending on the boundary condition in the domain $\Omega$ is as follows: We introduce a linear operator $\mathcal{A}=\mathcal{A}_{\alpha,\beta}$ with the domain $D(\mathcal{A})$ associated with the boundary condition in $L^2(\Omega)$ defined by
\begin{align}
\label{Laplacian}
D(\mathcal{A})&:=\Big\{f\in H^1_{\alpha,\beta}(\Omega)\,:\,\Delta f\in L^2(\Omega)\Big\}, \\
\mathcal{A}f&:=\Delta f \quad \text{for}\ \ f\in D(\mathcal{A}).\notag
\end{align}
We write $\mathcal{A}$ as $\Delta=\Delta^{\alpha,\beta}$ and call $\mathcal{A}$ to be the Dirichlet Laplacian if $\alpha=0$, the Neumann Laplacian if $\beta=0$ and the Robin Laplacian otherwise. It is known that $\mathcal{A}$ is $m$-dissipative operator with the dense domain if $\alpha\beta\ge 0$ (see \cite[Definition 2.2.2 and Proposition 2.6.1]{CaHa98} and \cite{ArendtUrban10}, for example).

\item We define a Hilbert space $E$ as
$$E=E_{\alpha,\beta}=E_{\alpha,\beta}(\Omega):=\big(H_{\alpha,\beta}^1(\Omega)\big)^{\mathtt{k}}\times \big(L^2(\Omega)\big)^{\mathtt{k}}$$
with $\mathtt{k}\ge 2$, which is called the energy space to the initial-boundary value problem (\ref{Equation_Main}).
\item 
Let $m \in \mathbb{N}\cup\{0\}$. For an interval $I \subset \mathbb{R}$, we introduce the space $\mathcal{C}^{m}(I,X)$ of $m$-times continuously differentiable functions from $I$ to $X$ with respect to the topology in $X$.
%, that is,
%\[
%    C^m(I,X):=\left\{u:I\rightarrow X\ ;\ \sup_{t\in I}\sum_{i=0}^m\left\|\frac{d^iu}{dt^i}(t)\right\|_{X}<\infty\right\}
%\]
\item
Let $p\in [1,\infty]$. We denote by $p'$ the H\"older conjugate of $p$. Moreover, for an interval $I=[0,T)\subset \mathbb{R}$ with $T>0$, we define the Lebesgue space $L^p(0,T;X)$ of all measurable functions from $I$ to $X$ endowed with the norm $\|\cdot\|_{L^p(0,T;X)}$ given by 
$$ \| u\| _{L^{p}\left(0,T;X\right)}:=\big\|\,\| u(t,\cdot)\| _{X}\big\|_{L^{p}(I)}. $$
\end{itemize}

\section{Main results}\label{Sec.Main}
In this section, we state our two main results. The first one (Theorem \ref{LWP_Prop}) gives local well-posedness to the initial-boundary value problem (\ref{Equation_Main}) in the energy space $E$ for an arbitrary data in $E$ in the case of $p_{\ell}\in [1,d/(d-2)]$ with $d\ge 3$ and $p_{\ell}\in [1,\infty)$ with $d=1,2$ for $\ell=1,2,\cdots,\mathtt{k}$. Here we say that well-posedness to (\ref{Equation_Main}) holds if existence, uniqueness of the solution and continuous dependence on the initial data are valid. The second one (Theorem \ref{Upper_Bound_Thoerem}) gives a small data blow-up result and upper estimates for the lifespan of the small solutions to the problem (\ref{Equation_Main}) for a suitable data when $\gamma_{\max}\ge d/2$ and $d\ge 2$.

\subsection{Large data local well-posedness in the energy space}

In order to state the large data local well-posedness result to the initial-boundary value problem (\ref{Equation_Main}), we convert the original problem (\ref{Equation_Main}) into the following form:
\begin{equation}\label{Equation_Vector}
\begin{cases}
\partial_tU(t,x)-\mathcal{B}_{\alpha,\beta}U(t,x)=\mathcal{N}\big(U\big)(t,x), & (t,x)\in (0,T) \times \Omega,\\
U(0,x)= \varepsilon U_0(x), &x\in \Omega,
\end{cases}
\end{equation}
where ${\bm u}:(0,T)\times \Omega\rightarrow {\bm u}(t,x):=\big(u_1(t,x),u_2(t,x),\cdots,u_{\mathtt{k}}(t,x)\big)$ stands for the $\mathtt{k}$-tuple of the unknown functions, and
\begin{align*}
U:(0,T)\times\Omega\rightarrow U(t,x):=& \big(\bm{u}(t,x),\partial_t\bm{u}(t,x)\big)^{\mathtt{t}} \\
=& \big(u_1(t,x),u_2(t,x),\cdots,u_{\mathtt{k}}(t,x), \partial_tu_1(t,x),\partial_tu_2(t,x),\cdots,\partial_tu_{\mathtt{k}}(t,x)\big)^{\mathtt{t}}
\end{align*}
is a new unknown function. Moreover, $\bm{u}_0:\Omega\rightarrow \bm{u}_0(x):=\big(u_{0,1}(x),u_{0,2}(x)\cdots,u_{0,\mathtt{k}}(x)\big)$ denotes the $\mathtt{k}$-tuple of the initial displacement and $\bm{u}_1:\Omega\rightarrow \bm{u}_1(x):=\big(u_{1,1}(x),u_{1,2}(x),\cdots,u_{1,\mathtt{k}}(x)\big)$ denotes the $\mathtt{k}$-tuple of the initial velocity, and $$U_0:\Omega\rightarrow U_0(x):=\big(\bm{u}_0(x),\bm{u}_1(x)\big)^{\mathtt{t}}$$
is a new given initial function. The linear operator $\mathcal{B}_{\alpha,\beta}$ on the energy space $E_{\alpha,\beta}$ is defined by
\begin{equation}
\label{m-disipative op}
\mathcal{B}=\mathcal{B}_{\alpha,\beta}:=
   \begin{pmatrix}
   0 & 1\\
   \Delta_{\alpha,\beta} & 0\\
   \end{pmatrix}
\end{equation}
with the dense domain
$$ D(\mathcal{B}):=\left\{(\bm{u},\bm{v})\in E_{\alpha,\beta}\,:\,\Delta_{\alpha,\beta}\bm{u}\in \big(L^2(\Omega)\big)^{\mathtt{k}},\bm{v}\in \big(H^1_{\alpha,\beta}(\Omega)\big)^{\mathtt{k}}\right\}, $$
where $\Delta_{\alpha,\beta}$ is the Laplace operator defined by (\ref{Laplacian}) dependent on the boundary condition.
Finally, the nonlinear mapping $\mathcal{N}$ is given by
\[
  \mathcal{N}(U):=\Big(0,0,\cdots,0,-\partial_tu_1+|u_{\mathtt{k}}|^{p_1},-\partial_tu_2+|u_1|^{p_2},\cdots,-\partial_tu_{\mathtt{k}}+|u_{\mathtt{k}-1}|^{p_{\mathtt{k}}}\Big)^{\mathtt{t}}.
\]
We remark that the original problem (\ref{Equation_Main}) is equivalent to the problem (\ref{Equation_Vector}) through the relation
$$U=\Big(u_1,u_2,\cdots,u_{\mathtt{k}}, \partial_tu_1,\partial_tu_2,\cdots,\partial_tu_{\mathtt{k}}\Big)^{\mathtt{t}}. $$
The $m$-dissipativity of the operator $\mathcal{B}$ is known in the following lemma.
\begin{lemma}[$m$-dissipativity of $\mathcal{B}$]
\label{m-disipativity}
Let $\alpha\beta\ge 0$. Then the operator $\mathcal{B}$ defined by \eqref{m-disipative op} is $m$-dissipative with the dense domain $D(\mathcal{B})$ in $E$. Moreover, $\mathcal{B}$ generates a contraction semigroup $\left\{e^{t\mathcal{B}}\right\}_{t\ge 0}$ on $E$.
\end{lemma}
This lemma can be proved in the similar manner as the proof of \cite[Proposition 2.6.9]{CaHa98} and \cite[Theorem 3.4.4]{CaHa98}.

For $T\in (0,\infty]$, we define a solution space $E(T)$ to the problem (\ref{Equation_Vector}) as $E(T):=L^{\infty}(0,T;E)$, which is a Banach space endowed with the norm $\|\cdot\|_{L^{\infty}(0,T;E)}$.
Next we introduce the following notions of mild solution and lifespan of solutions to the problems (\ref{Equation_Main}) as well as (\ref{Equation_Vector}).
\begin{definition}[Mild solution]
\label{defmildsol}
We say that a function ${\bm u}=(u_1,u_2,\cdots,u_\mathtt{k})$ is a mild solution to (\ref{Equation_Main}) if ${\bm u}$ possesses the regularity
$${\bm u}\in\Big(\mathcal{C}\big([0,T),H^1_{\alpha,\beta}(\Omega)\big) \cap \mathcal{C}^1\big([0,T),L^2(\Omega)\big)\Big)^{\mathtt{k}}$$
and it satisfies the following integral equation:
\begin{equation}
\label{Integraleq}
   U(t)=\varepsilon e^{t\mathcal{B}}U_0+\int_0^te^{(t-\tau)\mathcal{B}}\mathcal{N}\big(U\big)(\tau)d\tau,
\end{equation}
which is associated with (\ref{Equation_Vector}), for $U_0\in E$ and for any $t\in [0,T)$. Additionally, we call $U$ belonging to the class $E(T)$ a mild solution to (\ref{Equation_Vector}) if ${\bm u}$ is a mild solution to (\ref{Equation_Main}).
\end{definition}

\begin{definition}[Lifespan]
\label{deflifespan}
We call the maximal existence time of the mild solution to (\ref{Equation_Main}) to be lifespan, which is denoted by $T_{\varepsilon}$, that is
\[
T_{\varepsilon}
:=\sup\Big\{T\in (0,\infty]\ :\ \text{There exists a unique mild solution ${\bm u}$ to (\ref{Equation_Main}) on $[0,T)$}\Big\}.
\]
\end{definition}
Now let us state the first main result concerning the local well-posedness to (\ref{Equation_Vector}) in the energy space $E$ as follows:
\begin{theorem} [Large data local well-posedness in the energy space]
\label{LWP_Prop}
Let $\alpha,\beta\in \mathbb{R}$ satisfying $\alpha\beta\ge 0$, $d\in \N$, $p_{\ell}\ge 1$ for $\ell=1,2,\cdots,\mathtt{k}$, $\varepsilon>0$ and $U_0\in E$. We assume that if $d\ge 3$, then $p_{\ell}\le d/(d-2)$ for $\ell=1,2,\cdots,\mathtt{k}$. Then the initial-boundary value problem \eqref{Equation_Main} is locally well-posed in the energy space $E$. More precisely, the following statements hold:
\begin{itemize}[leftmargin=*]
\item {\rm\textbf{Existence}}: There exists a positive time $T=T\Big(\varepsilon,\|U_0\|_E,\big\{p_{\ell}\big\}_{\ell=1}^{\mathtt{k}},d\Big)>0$ such that there exists a mild solution $U\in E(T)\cap \mathcal{C}([0,T);E)$ to the problem \eqref{Equation_Vector} on $[0,T)$.
\item {\rm\textbf{Uniqueness}}: Let $U\in E(T)$ be the solution to \eqref{Equation_Vector} obtained in the Existence part. Let $T_1\in (0,T)$ and $V\in E(T_1)$ be another mild solution to \eqref{Equation_Vector} with the same initial data $\varepsilon U_0$. Then the identity $U(t)=V(t)$ holds for any $t\in [0,T_1]$.
\item {\rm\textbf{Continuous dependence on initial data}}: The flow map $E\rightarrow E(T,M)$,\ $\varepsilon U_0\mapsto U$ is Lipschitz continuous, where $U$ is the mild solution to \eqref{Equation_Vector} with initial data $\varepsilon U_0$ and the subspace $E(T,M)$ is determined as in \eqref{Def_E(T,M)}.
\item {\rm\textbf{Blow-up alternative}}: If $T_{\varepsilon}<\infty$, then $\dps\lim_{t\rightarrow T_{\varepsilon}-0}\|U(t,\cdot)\|_{E}=\infty$.
\end{itemize}
The above existence and uniqueness results imply well-definedness of the lifespan $T_{\varepsilon}$ given in Definition \ref{deflifespan}, that is, $T_{\varepsilon}>0$.
\end{theorem}
\subsection{Blow-up and lifespan estimates for small data solutions}
Before indicating the blow-up result for small data solutions as well as the lifespan estimates, let us state the following definition of weak solutions to \eqref{Equation_Main}.
\begin{definition}[Weak solution] \label{Weak_solutions_Def}
Let $p_\ell>1$ with $\ell=1,2,\cdots,\mathtt{k}$ and $T>0$. A $\mathtt{k}$-tuple of functions $(u_1,u_2,\cdots,u_\mathtt{k})$ is called a weak solution to \eqref{Equation_Main} on $[0,T)$ if
\begin{align*}
(u_1,u_2,\cdots,u_\mathtt{k})\in &\Big(\mathcal{C}\big([0,T),H^1_0(\Omega)\big) \cap \mathcal{C}^1\big([0,T),L^2(\Omega)\big) \cap L^{p_2}_{\rm loc}\big([0,T) \times \overline{\Omega}\big)\Big) \\ 
&\quad \times \Big(\mathcal{C}\big([0,T),H^1_0(\Omega)\big) \cap \mathcal{C}^1\big([0,T),L^2(\Omega)\big) \cap L^{p_3}_{\rm loc}\big([0,T) \times \overline{\Omega}\big)\Big) \\
&\qquad \ddots \\
&\qquad\quad \times \Big(\mathcal{C}\big([0,T),H^1_0(\Omega)\big) \cap \mathcal{C}^1\big([0,T),L^2(\Omega)\big) \cap L^{p_1}_{\rm loc}\big([0,T) \times \overline{\Omega}\big)\Big)
\end{align*}
if $\alpha=0$ or
\begin{align*}
(u_1,u_2,\cdots,u_\mathtt{k})\in &\Big(\mathcal{C}\big([0,T),H^1(\Omega)\big) \cap \mathcal{C}^1\big([0,T),L^2(\Omega)\big) \cap L^{p_2}_{\rm loc}\big([0,T) \times \overline{\Omega}\big)\Big) \\ 
&\quad \times \Big(\mathcal{C}\big([0,T),H^1(\Omega)\big) \cap \mathcal{C}^1\big([0,T),L^2(\Omega)\big) \cap L^{p_3}_{\rm loc}\big([0,T) \times \overline{\Omega}\big)\Big) \\
&\qquad \ddots \\
&\qquad\quad \times \Big(\mathcal{C}\big([0,T),H^1(\Omega)\big) \cap \mathcal{C}^1\big([0,T),L^2(\Omega)\big) \cap L^{p_1}_{\rm loc}\big([0,T) \times \overline{\Omega}\big)\Big),
\end{align*}
if $\alpha\ne 0$, respectively, and moreover the following relations hold:
\begin{align}
&\int_0^T \int_{\Omega}|u_{\mathtt{k}}(t,x)|^{p_1} \Phi(t,x)dxdt+ \int_{\Omega} u_{1,1}(x)\Phi(0,x)dx \nonumber \\
&\quad = \int_0^T \int_{\Omega}\Big(\nabla u_1(t,x) \cdot \nabla \Phi(t,x)- \partial_t u_1(t,x) \partial_t \Phi(t,x)+ \partial_t u_1(t,x) \Phi(t,x) \Big)dxdt \label{weaksolution1}
\end{align}
and
\begin{align}
&\int_0^T \int_{\Omega}|u_\ell(t,x)|^{p_{\ell+1}} \Phi(t,x)dxdt+ \int_{\Omega} u_{1,\ell+1}(x)\Phi(0,x)dx \nonumber \\
&\quad = \int_0^T \int_{\Omega}\Big(\nabla u_{\ell+1}(t,x) \cdot \nabla \Phi(t,x)- \partial_t u_{\ell+1}(t,x) \partial_t \Phi(t,x)+ \partial_t u_{\ell+1}(t,x) \Phi(t,x) \Big)dxdt, \label{weaksolution2}
\end{align}
with $\ell=1,2,\cdots,\mathtt{k}-1$, for any test function $\Phi= \Phi(t,x) \in \mathcal{C}^2\big([0,T)\times \Omega\big)$ with ${\rm supp}\Phi \subset [0,T)\times \overline{\Omega}$ such that $\left(\alpha\f{\Phi}{\partial n^+}+\beta\Phi\right)(t,\cdot)\Big|_{\partial \Omega}=0$.
\end{definition}

\begin{lemma}[Relation between mild solution and weak solution]
We assume the same assumptions as in Theorem \ref{LWP_Prop}. Then $(u_1,u_2,\cdots,u_{\mathtt{k}})$ is a mild solution to \eqref{Equation_Main} on $[0,T)$ in the sense of Definition \ref{defmildsol} if and only if it is a weak solution to \eqref{Equation_Main} on $[0,T)$ in the sense of Definition \ref{Weak_solutions_Def}, respectively.
\end{lemma}
The proof of this lemma is standard, thus we omit its detail.

The second main result is read as follows:
\begin{theorem}[Blow-up and lifespan estimates]
\label{Upper_Bound_Thoerem}
Let $d\ge 2$. With $\ell=1,2,\cdots,\mathtt{k}$, assume that the exponents $p_\ell>1$ fulfill the condition
\begin{equation} \label{Critical_Condition}
\max\{\gamma_1,\gamma_2,\cdots,\gamma_{\mathtt{k}}\}\ge \frac{d}{2},
\end{equation}
and the initial data $u_{0,\ell},\,u_{1,\ell}$ satisfy
\begin{equation}\label{Assumption1_Initial_data}
u_{0,\ell}\Psi+u_{1,\ell}\Psi \in L^1(\Omega) 
\end{equation}
as well as
\begin{equation}\label{Assumption2_Initial_data}
\int_{\Omega}\big(u_{0,\ell}(x)+ u_{1,\ell}(x)\big)\Psi(x)dx>0,
\end{equation}
where the function $\Psi= \Psi(x)$ is defined by
\begin{itemize}
\item $\beta\ne0$:
\begin{equation}\label{Test_function_x}
\Psi(x)=
\begin{cases}
\log|x|+\f{\alpha}{\beta} &\text{ if }\ \ d=2,\\
1-|x|^{2-d}+\f{\alpha}{\beta}(d-2) &\text{ if }\ \ d\ge 3,
\end{cases}
\end{equation}
\item $\beta=0$:
\begin{equation}\label{Test_function_x*}
\Psi(x)= 1\quad \text{ for all }d\ge 2.
\end{equation}
\end{itemize}
Then, there exist positive constants $\varepsilon_0>0$ and $C>0$ such that for any $\varepsilon \in (0,\varepsilon_0]$ the following upper bound estimates for the lifespan of weak solutions to \eqref{Equation_Main} hold:
\begin{itemize}[leftmargin=*]
\item $\beta\ne 0$:
\begin{align} %\label{Lifespan}
&{\rm LifeSpan}(u_1,u_2,\cdots,u_{\mathtt{k}}) \nonumber \\
&\qquad \le
\begin{cases}
C\left(\ep^{-1}\log\big(\ep^{-1}\big)\right)^{(\max\{\gamma_1,\gamma_2,\cdots,\gamma_{\mathtt{k}}\}- 1)^{-1}} &\text{if}\ \ \max\{\gamma_1,\gamma_2,\cdots,\gamma_{\mathtt{k}}\}> 1, \ \ d=2, \\
\exp\exp\left(C\ep^{-1}\right) &\text{if}\ \ \max\{\gamma_1,\gamma_2,\cdots,\gamma_{\mathtt{k}}\}= 1, \ \ d=2, \\
&\quad\,\, p_1=p_2=\cdots=p_{\mathtt{k}}, \\
C\varepsilon^{-(\max\{\gamma_1,\gamma_2,\cdots,\gamma_{\mathtt{k}}\}- d/2)^{-1}} &\text{if}\ \ \max\{\gamma_1,\gamma_2,\cdots,\gamma_{\mathtt{k}}\}> d/2, \ \ d\ge 3, \\
\exp\left(C\varepsilon^{-(p_1p_2\cdots p_{\mathtt{k}}-1)}\right) &\text{if}\ \ \max\{\gamma_1,\gamma_2,\cdots,\gamma_{\mathtt{k}}\}= d/2,\ \ d\ge 3, \\
&\quad p_{j_1}\ne p_{j_2} \text{ with }j_1,j_2\in \{1,2,\cdots,\mathtt{k}\} \text{ and }j_1\ne j_2, \\
\exp\left(C\varepsilon^{-(p_1-1)}\right) &\text{if}\ \ \max\{\gamma_1,\gamma_2,\cdots,\gamma_{\mathtt{k}}\}= d/2, \ \ d\ge 3, \\
&\quad\,\, p_1=p_2=\cdots=p_{\mathtt{k}},
\end{cases} \nonumber
\end{align}
\item $\beta=0$:
\begin{align} %\label{Lifespan}
&{\rm LifeSpan}(u_1,u_2,\cdots,u_{\mathtt{k}}) \nonumber \\
&\qquad \le
\begin{cases}
C\varepsilon^{-(\max\{\gamma_1,\gamma_2,\cdots,\gamma_{\mathtt{k}}\}- d/2)^{-1}} &\text{if}\ \ \max\{\gamma_1,\gamma_2,\cdots,\gamma_{\mathtt{k}}\}> d/2, \ \ d\ge 2, \\
\exp\left(C\varepsilon^{-(p_1p_2\cdots p_{\mathtt{k}}-1)}\right) &\text{if}\ \ \max\{\gamma_1,\gamma_2,\cdots,\gamma_{\mathtt{k}}\}= d/2,\ \ d\ge 2, \\
&\quad p_{j_1}\ne p_{j_2} \text{ with }j_1,j_2\in \{1,2,\cdots,\mathtt{k}\} \text{ and }j_1\ne j_2, \\
\exp\left(C\varepsilon^{-(p_1-1)}\right) &\text{if}\ \ \max\{\gamma_1,\gamma_2,\cdots,\gamma_{\mathtt{k}}\}= d/2, \ \ d\ge 2, \\
&\quad\,\, p_1=p_2=\cdots=p_{\mathtt{k}},
\end{cases} \nonumber
\end{align}
\end{itemize}
where $C$ is a positive constant independent of $\ep$.
\end{theorem}

\begin{remark}
By plugging $\mathtt{k}=1$ in Theorem \ref{Upper_Bound_Thoerem}, one may realize obviously that our main results from Theorem \ref{Upper_Bound_Thoerem} exactly coincide with (\ref{Lifespan_Single-Eq}) derived from \cite{Sobajima2019-1,IkedaSobajima2019-1} when the case of the Dirichlet boundary condition, i.e. $\alpha=0$, is considered.
\end{remark}

\begin{remark}
This remark is to underline that in the present paper we have succeeded not only in proving the local well-posedness in the energy space and the blow-up phenomenon of small data solutions but also in catching upper bound estimates for lifespan of solutions to \eqref{Equation_Main}. By using an appropriately weighted energy method, we will devote to our concern in investigating global existence results for solutions to \eqref{Equation_Main} in a forthcoming paper.
\end{remark}

\begin{remark}
As we can see in Theorem \ref{Upper_Bound_Thoerem}, an open problem, which should be recognized to explore in a further study, is to find out a suitable upper bound in estimating the lifespan of solutions to \eqref{Equation_Main} with $\beta\ne 0$ in two-dimensional case, where the critical case is of interest and the exponents $p_1,p_2,\cdots,p_{\mathtt{k}}$ are not necessarily equal.
\end{remark}

\section{Proof of Theorem \ref{LWP_Prop}}\label{Sec.Theorem-1}
In this section, we give a proof of local well-posedness in the energy space to the problems (\ref{Equation_Main}) (Theorem \ref{LWP_Prop}). For $U_0\in E$ and $T>0$, we introduce a nonlinear mapping $\mathcal{J}$ defined by
\begin{equation}
\label{Integral eq}
   \mathcal{J}[U](t):=\varepsilon e^{t\mathcal{B}}U_0+\int_0^te^{(t-\tau)\mathcal{B}}\mathcal{N}\big(U\big)(\tau)d\tau
\end{equation}
for $t\in [0,T)$. To construct a local mild solution to (\ref{Equation_Vector}), we will prove that $\mathcal{J}$ is a contraction mapping from a suitable closed set in $E(T)$ into itself. We recall the following Sobolev embedding.
\begin{lemma}[Sobolev embedding]
\label{Sobolev embedding}
Let $\Omega$ be an open subset of $\R^d$ which has a Lipschitz continous boundary. Let ``$p\in [1,d)$ and $q\in [p,pd/(d-p)]$'' or ``$p=d$ and $q\in [p,\infty)$''. Then the embedding $W^{1,p}(\Omega)\hookrightarrow L^q(\Omega)$ holds. 
\end{lemma}
For the proof of this lemma, see \cite{Adams75} or \cite[Lemma 3.3]{IkedaTaniguchiWakasugiArxiv}.

Now we give a proof of Theorem \ref{LWP_Prop}.
\begin{proof}[Proof of Theorem \ref{LWP_Prop}]
We prove the Existence part only, since the other parts can be proved in a similar or standard manner. Let $M\ge 2\varepsilon \|U_0\|_E$. We take $T>0$ such as
$$
    T\le \frac{1}{4}\min\left\{1,\frac{1}{2C_*\sum_{\ell=1}^{\mathtt{k}}M^{p_{\ell}-1}}\right\},
$$
where $C_*>0$ is defined in (\ref{definitionC_*}) and (\ref{definitionC_*2}) below and independent of $T$. We introduce a closed ball $E(T,M)$ at the origin with radius $M$ in the Banach space $E(T)$ given by
\begin{equation} \label{Def_E(T,M)}
    E(T,M):=\left\{U\in E(T)\ :\ \|U\|_{E(T)}\le M\right\}
\end{equation}
with a metric $d_T:E(T)\times E(T)\rightarrow \R_{\ge 0}$ defined by
$$
  d_T(U_1,U_2):=\|U_1-U_2\|_{E(T)}.
$$
We prove that the nonlinear mapping $\mathcal{J}$ given by (\ref{Integral eq}) is a contraction mapping from $E(T,M)$ into itself. 
%By Lemma \ref{m-disipativity}, the estimates
%\begin{equation}
%    \|\varepsilon e^{tB}U_0\|_{E(T)}\le \varepsilon\|U_0\|_{E}\le \varepsilon r
%\end{equation}
%hold. 
Let $U\in E(T,M)$. By the definitions of the nonlinear function $\mathcal{N}$ and the energy space $E$ and the Sobolev embedding $H^1(\Omega)\hookrightarrow L^{2p_{\ell}}(\Omega)$ (Lemma \ref{Sobolev embedding}) for $\ell=1,2,\cdots,\mathtt{k}$, the following estimates:
\begin{align}
    \|\mathcal{N}(U)(t,\cdot)\|_E&=\big\|\big(-\partial_tu_1(t,\cdot)+|u_{\mathtt{k}}(t,\cdot)|^{p_1},-\partial_tu_2(t,\cdot)+|u_1(t,\cdot)|^{p_2},\cdots,-\partial_tu_{\mathtt{k}}(t,\cdot)+|u_{\mathtt{k}-1}(t,\cdot)|^{p_{\mathtt{k}}}\big)\big\|_{\big(L^2(\Omega)\big)^{\mathtt{k}}}\notag\\
    &\le \sum_{\ell=1}^{\mathtt{k}}\big\|\partial_tu_{\ell}(t,\cdot)\big\|_{L^2(\Omega)}+\sum_{\ell=1}^{\mathtt{k}}\|u_{\ell-1}(t,\cdot)\|_{L^{2p_{\ell}}}^{p_{\ell}} \notag\\
    &\le \|U\|_{E(T)}+C_*\sum_{\ell=1}^{\mathtt{k}}\|U\|_{E(T)}^{p_{\ell}}\le M +C_*\sum_{\ell=1}^{\mathtt{k}}M^{p_{\ell}}
\label{definitionC_*}
\end{align}
hold for some positive constant $C_*$ independent of $T$, where we set $u_0:=u_{\mathtt{k}}$. By Lemma \ref{m-disipativity} and the estimates (\ref{definitionC_*}), one derives
\begin{align*}
    \|\mathcal{J}[U](t)\|_{E}&\le \big\|\varepsilon e^{t\mathcal{B}}U_0\big\|_{E}+\int_0^t\big\|e^{(t-\tau)\mathcal{B}}\mathcal{N}\big(U\big)(\tau)\big\|_Ed\tau
    \le \varepsilon\|U_0\|_E+T\sup_{t\in [0,T)}\|\mathcal{N}\big(U\big)(t,\cdot)\|_{E}\\
    &\le \frac{M}{2}+TM+C_*T\sum_{\ell=1}^{\mathtt{k}}M^{p_{\ell}}
    \le \frac{M}{2}+\frac{M}{4}+\frac{M}{8}\le M,
\end{align*}
which implies that the mapping $\mathcal{J}$ from $E(T,M)$ to itself is well-defined. Let $U,V\in E(T,M)$. We set
$$U:=\big(u_1,u_2,\cdots,u_{\mathtt{k}}, \partial_tu_1,\partial_tu_2,\cdots,\partial_tu_{\mathtt{k}}\big)\,\,\text{ and }\,\, V:=\big(v_1,v_2,\cdots,v_{\mathtt{k}}, \partial_tv_1,\partial_tv_2,\cdots,\partial_tv_{\mathtt{k}}\big).$$
By the definitions of the nonlinear function $\mathcal{N}$ and the energy space $E$, the H\"older inequality and the Sobolev embedding $H^1(\Omega)\hookrightarrow L^{2p_{\ell}}(\Omega)$ (Lemma \ref{Sobolev embedding}) for $\ell=1,2,\cdots,\mathtt{k}$, the following chain of the inequalities:
\begin{align}
    &\big\|\mathcal{N}(U)(t,\cdot)-\mathcal{N}(V)(t,\cdot)\big\|_E \notag\\
    &\qquad\le \big\|\big(-\partial_tu_1(t,\cdot)+\partial_tv_1(t,\cdot)\big)+\big(|u_{\mathtt{k}}(t,\cdot)|^{p_1}-|v_{\mathtt{k}}(t,\cdot)|^{p_1}\big)\big\|_{L^2(\Omega)}\notag\\
    &\qquad\quad +\big\|\big(-\partial_tu_2(t,\cdot)+\partial_tv_2(t,\cdot)\big)+\big(|u_1(t,\cdot)|^{p_2}-|v_1(t,\cdot)|^{p_2}\big)\big\|_{L^2(\Omega)}\notag\\
    &\qquad\qquad \ddots \notag\\
    &\qquad\qquad\quad +\big\|\big(-\partial_tu_{\mathtt{k}}(t,\cdot)+\partial_tv_{\mathtt{k}}(t,\cdot)\big)+\big(|u_{\mathtt{k}-1}(t,\cdot)|^{p_{\mathtt{k}}}-|v_{\mathtt{k}-1}(t,\cdot)|^{p_{\mathtt{k}}}\big)\big\|_{L^2(\Omega)}\notag\\
    &\qquad\le \sum_{\ell=1}^{\mathtt{k}}\big\|\partial_tu_{\ell}(t,\cdot)-\partial_tv_{\ell}(t,\cdot)\big\|_{L^2(\Omega)}+\sum_{\ell=1}^{\mathtt{k}}\big\||u_{\ell-1}(t,\cdot)|^{p_{\ell}}-|v_{\ell-1}(t,\cdot)|^{p_{\ell}}\big\|_{L^2(\Omega)}\notag \\
    &\qquad\le \|U-V\|_{E(T)}+C\sum_{\ell=1}^{\mathtt{k}}\left(\|u_{\ell-1}(t,\cdot)\|_{L^{2p_{\ell}}(\Omega)}^{p_{\ell}-1}+\|v_{\ell-1}(t,\cdot)\|_{L^{2p_{\ell}}(\Omega)}^{p_{\ell}-1}\right)\|u_{\ell-1}(t,\cdot)-v_{\ell-1}(t,\cdot)\|_{L^{2p_{\ell}}(\Omega)}\notag \\
    &\qquad\le d_T(U,V)+C_*\sum_{\ell=1}^{\mathtt{k}}\Big(\|U\|_{E(T)}^{p_{\ell}-1}+\|V\|_{E(T)}^{p_{\ell}-1}\Big)\|U-V\|_{E(T)}\notag \\
    &\qquad\le d_T(U,V)+2C_*\sum_{\ell=1}^{\mathtt{k}}M^{p_{\ell}-1}d_T(U,V)
\label{definitionC_*2}
\end{align}
holds, where we set $v_0:=v_{\mathtt{k}}$. By Lemma \ref{Sobolev embedding} and the inequality (\ref{definitionC_*2}), we obtain
\begin{align*}
    d_T\big(\mathcal{J}[U],\mathcal{J}[V]\big)&=\sup_{t\in [0,T)}\big\|\mathcal{J}[U](t)-\mathcal{J}[V](t)\big\|_E
    \le
    \sup_{t\in [0,T)}\int_0^t\Big\|e^{(t-\tau)\mathcal{B}}\Big(\mathcal{N}\big(U\big)(\tau)-\mathcal{N}\big(V\big)(\tau)\Big)\Big\|_Ed\tau\\
    &\le T\sup_{t\in [0,T)}\big\|\mathcal{N}(U)(t,\cdot)-\mathcal{N}(V)(t,\cdot)\big\|_E
    \le Td_T(U,V)+2C_*T\sum_{\ell=1}^{\mathtt{k}}M^{p_{\ell}-1}d_T(U,V)\\
    &\le \frac{1}{2}d_T(U,V),
\end{align*}
which implies that the mapping $\mathcal{J}$ is a contraction mapping. By the contraction mapping principle, we see that there exists a unique function $U\in E(T,M)$ such that the identity $\mathcal{J}[U](t)=U(t)$ holds for any $t\in [0,T)$. By a standard argument, we can prove that $U\in \mathcal{C}([0,T);E)$, which completes the proof of the theorem.
\end{proof}

\section{Proof of Theorem \ref{Upper_Bound_Thoerem}}\label{Sec.Theorem-2}

\subsection{Test function method} \label{Test_functions}
At first, let us introduce a test function $\varphi=\varphi(\rho)$ having the following properties:
\begin{align*}
\varphi\in \mathcal{C}_0^{\infty}\big([0,\infty)\big)\ \ \mbox{and}\ \ \varphi(\rho):=
\begin{cases}
1&\mbox{if}\ \ \rho\in[0,1/2],\\
\mbox{decreasing}&\mbox{if}\ \ \rho\in(1/2,1),\\
0&\mbox{if}\ \ \rho\in[1,\infty).
\end{cases}
\end{align*}
Then, another test function $\varphi^*=\varphi^*(\rho)$ is given by
\begin{align*}
 \varphi^*(\rho):=
\begin{cases}
0&\mbox{if}\ \ \rho\in[0,1/2),\\
\varphi(\rho)&\mbox{if}\ \ \rho\in[1/2,\infty).
\end{cases}
\end{align*}
Let $R\in(0,\infty)$ be a large parameter. We introduce two test functions $\phi_R=\phi_R(t,x)$ and $\phi_R^*=\phi_R^*(t,x)$ as follows:
\begin{align*}
\phi_R(t,x):=\left(\varphi\left(\frac{t^2+(|x|-1)^4}{R^4}\right)\right)^{\lambda+2}\ \ \mbox{and}\ \ \phi_R^*(t,x):=\left(\varphi^*\left(\frac{t^2+(|x|-1)^4}{R^4}\right)\right)^{\lambda+2}
\end{align*}
with a positive constant $\lambda$, which will be fixed later. In addition, we define two notations
\begin{align*}
Q_R:= &\left\{(t,x) \in (0,T)\times \Omega \,:\, t^2+(|x|-1)^4 < R^4 \right\}, \\ 
Q^*_R:= &\left\{(t,x) \in (0,T)\times \Omega \,:\, \frac{R^4}{2}< t^2+(|x|-1)^4 < R^4 \right\}.
\end{align*}
\par The following useful lemma comes into play in our proof in the next section.
\begin{lemma}\label{Useful_Lemma}
The following estimates hold for any $(t,x) \in Q_R$:
\begin{align*}
{\rm (i)}&\qquad \big|\partial_t \phi_R(t,x)\big| \lesssim R^{-2}\big(\phi_R^*(t,x)\big)^{\frac{\lambda+1}{\lambda+2}}, \\
{\rm (ii)}&\qquad \big|\partial^2_t \phi_R(t,x)\big| \lesssim R^{-4}\big(\phi_R^*(t,x)\big)^{\frac{\lambda}{\lambda+2}}, \\
{\rm (iii)}&\qquad \big|\Delta \phi_R(t,x)\big| \lesssim R^{-2}\big(\phi_R^*(t,x)\big)^{\frac{\lambda}{\lambda+2}}.
\end{align*}
Moreover, by taking $\Psi= \Psi(x)$ as in \eqref{Test_function_x} or \eqref{Test_function_x*} we have the further estimate as follows:
\begin{align*}
{\rm (iv)}&\qquad \big|\Delta \big(\Psi(x) \phi_R(t,x)\big)\big| \lesssim R^{-2}\Psi(x)\big(\phi_R^*(t,x)\big)^{\frac{\lambda}{\lambda+2}}.
\end{align*}
\end{lemma}
\begin{proof}
First of all, a direct calculation leads to
\begin{align*}
	\partial_t\phi_R(t,x)&= \frac{2(\lambda+2)}{R^4}t\left(\varphi\left(\frac{t^2+(|x|-1)^4}{R^4}\right)\right)^{\lambda+1}\varphi'\left(\frac{t^2+(|x|-1)^4}{R^4}\right),\\
	\partial_t^2\phi_R(t,x)&= \frac{2(\lambda+2)}{R^4}\left(\varphi\left(\frac{t^2+(|x|-1)^4}{R^4}\right)\right)^{\lambda+1}\varphi'\left(\frac{t^2+(|x|-1)^4}{R^4}\right)\\
	&\quad+ \frac{4(\lambda+1)(\lambda+2)}{R^8}t^2\left(\varphi\left(\frac{t^2+(|x|-1)^4}{R^4}\right)\right)^{\lambda}\left(\varphi'\left(\frac{t^2+(|x|-1)^4}{R^4}\right)\right)^2\\
	&\quad+ \frac{4(\lambda+2)}{R^8}t^2\left(\varphi\left(\frac{t^2+(|x|-1)^4}{R^4}\right)\right)^{\lambda+1}\varphi''\left(\frac{t^2+(|x|-1)^4}{R^4}\right),
\end{align*}
and
\begin{align*}
	\nabla \phi_R(t,x)&= \frac{4(\lambda+2)}{R^4}(|x|-1)^3\frac{x}{|x|}\left(\varphi\left(\frac{t^2+(|x|-1)^4}{R^4}\right)\right)^{\lambda+1}\varphi'\left(\frac{t^2+(|x|-1)^4}{R^4}\right),\\
	\Delta \phi_R(t,x)&= \frac{12(\lambda+2)}{R^4}(|x|-1)^2 \left(\varphi\left(\frac{t^2+(|x|-1)^4}{R^4}\right)\right)^{\lambda+1}\varphi'\left(\frac{t^2+(|x|-1)^4}{R^4}\right)\\
	&\quad+ \frac{16(\lambda+1)(\lambda+2)}{R^8}(|x|-1)^6\left(\varphi\left(\frac{t^2+(|x|-1)^4}{R^4}\right)\right)^{\lambda}\left(\varphi'\left(\frac{t^2+(|x|-1)^4}{R^4}\right)\right)^2\\
	&\quad+ \frac{16(\lambda+2)}{R^8}(|x|-1)^6\left(\varphi\left(\frac{t^2+(|x|-1)^4}{R^4}\right)\right)^{\lambda+1}\varphi''\left(\frac{t^2+(|x|-1)^4}{R^4}\right).
\end{align*}
Thanks to the auxiliary properties
\begin{align*}
	\varphi'\left(\frac{t^2+(|x|-1)^4}{R^4}\right)\not\equiv0,\ \ \varphi''\left(\frac{t^2+(|x|-1)^4}{R^4}\right)\not\equiv0 \ \ \mbox{and}\ \  |x|-1\le R
\end{align*}
for any $(t,x) \in Q^*_R$, we may conclude the estimates from ${\rm (i)}$ to ${\rm (iii)}$. Next, in order to verify the last estimate, one observes that
\begin{align*}
\Delta \big(\Psi(x) \phi_R(t,x)\big) &= \Delta \Psi(x) \phi_R(t,x)+ \nabla \Psi(x) \cdot \nabla \phi_R(t,x)+ \Psi(x) \Delta \phi_R(t,x)\\ 
&= \nabla \Psi(x) \cdot \nabla \phi_R(t,x)+ \Psi(x) \Delta \phi_R(t,x)
\end{align*}
because of the fact $\Delta \Psi(x)=0$. Let us divide our consideration into two cases as follows:
\begin{itemize}[leftmargin=*]
\item[$\bullet$] If $\beta\ne 0$, then we take $\Psi(x)$ as in \eqref{Test_function_x}. By noticing that
$$\,\,\nabla \Psi(x)= 
\begin{cases}
\f{x}{|x|^2} &\text{ if }\ \ d=2, \\
(d-2)\f{x}{|x|^d} &\text{ if }\ \ d\ge 3,
\end{cases}$$
we derive
$$ \nabla \Psi(x) \cdot \nabla \phi_R(t,x)=
\begin{cases}
\dps\frac{4(\lambda+2)}{R^4}\cdot\frac{(|x|-1)^3}{|x|}\left(\varphi\left(\frac{t^2+(|x|-1)^4}{R^4}\right)\right)^{\lambda+1}\varphi'\left(\frac{t^2+(|x|-1)^4}{R^4}\right) &\text{ if }\ \ d=2, \\
\\
\dps\frac{4(d-2)(\lambda+2)}{R^4}\cdot\frac{(|x|-1)^3}{|x|^{n-1}}\left(\varphi\left(\frac{t^2+(|x|-1)^4}{R^4}\right)\right)^{\lambda+1}\varphi'\left(\frac{t^2+(|x|-1)^4}{R^4}\right) &\text{ if }\ \ d\ge 3.
\end{cases} $$
Thus, it implies
$$ \big|\nabla \Psi(x) \cdot \nabla \phi_R(t,x)\big|\lesssim R^{-2}\Psi(x)\big(\phi_R^*(t,x)\big)^{\frac{\lambda+1}{\lambda+2}} $$
by applying the elementary inequalities
$$ \begin{cases}
1-\f{1}{|x|}\le \log|x| &\text{ if }\ \ d=2, \\
|x|^{d-1}+1\ge 2|x| &\text{ if }\ \ d\ge 3,
\end{cases}$$
for any $|x|\ge 1$. 
\item[$\bullet$] If $\beta= 0$, then we take $\Psi(x)$ as in \eqref{Test_function_x*}. Obviously, one sees that $\nabla \Psi(x)= 0$, which entails immediately
$$\Delta \big(\Psi(x) \phi_R(t,x)\big)= \Psi(x) \Delta \phi_R(t,x). $$
\end{itemize}
Finally, using again the estimate ${\rm (iii)}$, it is clear to obtain the estimate ${\rm (iv)}$. Hence, our proof is complete.
\end{proof}

In the following proof, we will utilize the test functions as well as the notations defined in Section \ref{Test_functions}.

\subsection{Proof of Theorem \ref{Upper_Bound_Thoerem}} \label{Sec_Proof_Upper}
At first, we introduce the following test function:
$$ \Phi_R= \Phi_R(t,x):= \Psi(x) \phi_R(t,x), $$
which enjoys the conditions
$$ \Phi_R \in \mathcal{C}^2\big([0,T)\times \Omega\big), \quad {\rm supp}\Phi_R \subset [0,T)\times \overline{\Omega}\quad \text{ and }\quad \left(\alpha\f{\Phi_R}{\partial n^+}+\beta\Phi_R\right)(t,\cdot)\Big|_{\partial \Omega}=0. $$
We define the following functionals with $\ell=1,2,\cdots,\mathtt{k}-1$:
$$ \mathcal{I}_R[u_\ell]:= \int_0^{T}\int_{\Omega} |u_\ell(t,x)|^{p_{\ell+1}} \Phi_R(t,x)\,dxdt= \int_{Q_R}|u_\ell(t,x)|^{p_{\ell+1}} \Phi_R(t,x)\,dxdt $$
and
$$ \mathcal{I}_R[u_{\mathtt{k}}]:= \int_0^{T}\int_{\Omega} |u_{\mathtt{k}}(t,x)|^{p_1} \Phi_R(t,x)\,dxdt= \int_{Q_R}|u_{\mathtt{k}}(t,x)|^{p_1} \Phi_R(t,x)\,dxdt. $$
Let us assume that $(u_1,u_2,\cdots,u_\mathtt{k})= \big(u_1(t,x),u_2(t,x),\cdots,u_\mathtt{k}(t,x)\big)$ is a weak solution to \eqref{Equation_Main} in the sense of Definition \ref{Weak_solutions_Def}. By plugging $\Phi(t,x)= \Phi_R(t,x)$ into (\ref{weaksolution1}) and (\ref{weaksolution2}), taking integration by parts then we obtain
\begin{align}
&\mathcal{I}_R[u_{\mathtt{k}}]+ \varepsilon \int_{\Omega}\big(u_{0,1}(x)+u_{1,1}(x)\big)\Phi_R(0,x)dx \nonumber \\
&\qquad\quad= \int_{Q_R}u_1(t,x)\left(\partial_t^2\Phi_R(t,x)-\Delta\Phi_R(t,x)-\partial_t\Phi_R(t,x)\right)d(t,x), \label{proof.1}
\end{align}
and
\begin{align}
&\mathcal{I}_R[u_\ell]+ \varepsilon \int_{\Omega}\big(u_{0,\ell+1}(x)+u_{1,\ell+1}(x)\big)\Phi_R(0,x)dx \nonumber \\
&\qquad\quad= \int_{Q_R}u_{\ell+1}(t,x)\left(\partial_t^2\Phi_R(t,x)-\Delta\Phi_R(t,x)-\partial_t\Phi_R(t,x)\right)d(t,x) \label{proof.2}
\end{align}
with $\ell=1,2,\cdots,\mathtt{k}-1$, respectively. By the assumption (\ref{Assumption1_Initial_data}), the Lebesgue convergence theorem implies that
\begin{align*}
\lim_{R\to \ity}\int_{\Omega}\big(u_{0,\ell}(x)+u_{1,\ell}(x)\big)\Phi_R(0,x)dx &= \lim_{R\to \ity}\int_{\Omega}\big(u_{0,\ell}(x)+u_{1,\ell}(x)\big)\Psi(x)\phi_R(0,x)dx \\ 
&= \int_{\Omega}\big(u_{0,\ell}(x)+u_{1,\ell}(x)\big)\Psi(x)dx
\end{align*}
for any $\ell=1,2,\cdots,\mathtt{k}$. Together with the conditions \eqref{Assumption1_Initial_data} and \eqref{Assumption2_Initial_data} for the initial data, we deduce that there exists a sufficiently large constant $R_0$ so that it holds
\begin{align*}
\int_{\Omega}\big(u_{0,\ell}(x)+u_{1,\ell}(x)\big)\Phi_R(0,x)dx &\ge C^0_\ell
\end{align*}
for any $R>R_0$, where $C^0_\ell$ are suitable positive, small constants with $\ell=1,2,\cdots,\mathtt{k}$. Thus, it entails immediately from \eqref{proof.1} and \eqref{proof.2} that
$$ \mathcal{I}_R[u_{\mathtt{k}}]+ C^0_1 \varepsilon \le \int_{Q_R}u_1(t,x)\left(\partial_t^2\Phi_R(t,x)-\Delta\Phi_R(t,x)-\partial_t\Phi_R(t,x)\right)d(t,x) $$
and
$$ \mathcal{I}_R[u_\ell]+ C^0_{\ell+1} \varepsilon \le \int_{Q_R}u_{\ell+1}(t,x)\left(\partial_t^2\Phi_R(t,x)-\Delta\Phi_R(t,x)-\partial_t\Phi_R(t,x)\right)d(t,x) $$
with $\ell=1,2,\cdots,\mathtt{k}-1$. Employing Lemma \ref{Useful_Lemma} gives the following estimate:
\begin{align*}
\left|\partial_t^2\Phi_R(t,x)-\Delta\Phi_R(t,x)-\partial_t\Phi_R(t,x)\right| &= \left|\Psi(x)\partial_t^2\phi_R(t,x)- \Delta\big(\Psi(x)\phi_R(t,x)\big)- \Psi(x)\partial_t\phi_R(t,x)\right| \\
&\lesssim R^{-2}\Psi(x)\big(\phi_R^*(t,x)\big)^{\frac{\lambda}{\lambda+2}}
\end{align*}
due to the relation $0<\phi_R^*(t,x)< 1$ for any $R>R_0$. Hence, one achieves
$$ \mathcal{I}_R[u_{\mathtt{k}}]+ C^0_1 \varepsilon \lesssim R^{-2}\int_{Q_R}|u_1(t,x)|\,\Psi(x)\, \big(\phi_R^*(t,x)\big)^{\frac{\lambda}{\lambda+2}}d(t,x). $$
The application of H\"{o}lder inequality implies
\begin{align*}
\mathcal{I}_R[u_{\mathtt{k}}]+ C^0_1 \varepsilon &\lesssim R^{-2}\left(\int_{Q_R^*}\Psi(x)d(t,x)\right)^{\frac{1}{p'_2}}\left(\int_{Q_R}|u_1(t,x)|^{p_2}\,\Psi(x)\,(\phi_R^*(t,x))^{\frac{p_2\lambda}{\mu+2}}d(t,x)\right)^{\frac{1}{p_2}} \\
&\lesssim \Theta_{p_2}(R)\left(\dps\int_{Q_R}|u_1(t,x)|^{p_2}\,\Psi(x)\,(\phi_R^*(t,x))^{\frac{p_2\lambda}{\lambda+2}}d(t,x)\right)^{\frac{1}{p_2}},
\end{align*}
where the function $\Theta_{p_2}=\Theta_{p_2}(R)$ is defined by
\begin{equation}\label{Auxiliary_Funct}
\Theta_{p_2}(R)=
\begin{cases}
R^{2-\frac{4}{p_2}}(\log R)^{1-\frac{1}{p_2}} &\text{ if } \ \ d=2\,\,\text{ and }\,\,\beta\ne 0, \\
R^{d-\frac{d+2}{p_2}} &\text{ if } \ \ d\ge 3\,\,\text{ and }\,\,\beta\ne 0, \\
R^{d-\frac{d+2}{p_2}} &\text{ if } \ \ d\ge 2\,\,\text{ and }\,\,\beta= 0.
\end{cases}
\end{equation}
Here we notice that to derive the previous inequality, we have used the following estimate:
$$ \int_{Q_R^*}\Psi(x)d(t,x)\lesssim 
\begin{cases}
R^2(R+1)^2\log(R+1)\approx R^4\log R &\text{ if }\ \ d=2\,\,\text{ and }\,\,\beta\ne 0, \\
R^2(R+1)^d\approx R^{d+2} &\text{ if }\ \ d\ge 3\,\,\text{ and }\,\,\beta\ne 0, \\
R^2(R+1)^d\approx R^{d+2} &\text{ if } \ \ d\ge 2\,\,\text{ and }\,\,\beta= 0.
\end{cases}$$
for any $R>R_0$. Carrying out a similar way one finds
\begin{align*}
\mathcal{I}_R[u_1]+ C^0_2 \varepsilon &\lesssim \Theta_{p_3}(R)\left(\int_{Q_R}|u_2(t,x)|^{p_3}\,\Psi(x)\,(\phi_R^*(t,x))^{\frac{p_3\lambda}{\lambda+2}}d(t,x)\right)^{\frac{1}{p_3}}, \\
\mathcal{I}_R[u_2]+ C^0_3 \varepsilon &\lesssim \Theta_{p_4}(R)\left(\int_{Q_R}|u_3(t,x)|^{p_4}\,\Psi(x)\,(\phi_R^*(t,x))^{\frac{p_4\lambda}{\lambda+2}}d(t,x)\right)^{\frac{1}{p_4}}, \\
\quad \vdots & \nonumber \\
\mathcal{I}_R[u_{\mathtt{k}-1}]+ C^0_\mathtt{k} \varepsilon &\lesssim \Theta_{p_1}(R)\left(\int_{Q_R}|u_\mathtt{k}(t,x)|^{p_1}\,\Psi(x)\,(\phi_R^*(t,x))^{\frac{p_1\lambda}{\lambda+2}}d(t,x)\right)^{\frac{1}{p_1}}.
\end{align*}
Let us now choose the parameter $\lambda$ fulfilling
$$ \lambda\geqslant \max_{1\le \ell\le \mathtt{k}} \frac{2}{p_\ell-1}= \frac{2}{\dps\min_{1\le \ell\le \mathtt{k}}p_\ell-1}, \quad \text{ i.e. }\quad \min_{1\le \ell\le \mathtt{k}}p_\ell \frac{\lambda}{\lambda +2}\geqslant 1 $$
so that we may arrive at the following relations:
\begin{align}
\mathcal{I}_R[u_{\mathtt{k}}]+ C^0_1 \varepsilon &\lesssim \Theta_{p_2}(R)\left(\int_{Q_R}|u_1(t,x)|^{p_2}\,\Psi(x)\,\phi_R^*(t,x)d(t,x)\right)^{\frac{1}{p_2}}, \nonumber \\
\mathcal{I}_R[u_1]+ C^0_2 \varepsilon &\lesssim \Theta_{p_3}(R)\left(\int_{Q_R}|u_2(t,x)|^{p_3}\,\Psi(x)\,\phi_R^*(t,x)d(t,x)\right)^{\frac{1}{p_3}}, \nonumber \\
\mathcal{I}_R[u_2]+ C^0_3 \varepsilon &\lesssim \Theta_{p_4}(R)\left(\int_{Q_R}|u_3(t,x)|^{p_4}\,\Psi(x)\,\phi_R^*(t,x)d(t,x)\right)^{\frac{1}{p_4}}, \nonumber \\
\quad \vdots & \nonumber \\
\mathcal{I}_R[u_{\mathtt{k}-1}]+ C^0_\mathtt{k} \varepsilon &\lesssim \Theta_{p_1}(R)\left(\int_{Q_R}|u_\mathtt{k}(t,x)|^{p_1}\,\Psi(x)\,\phi_R^*(t,x)d(t,x)\right)^{\frac{1}{p_1}}. \label{proof.3}
\end{align}
Without loss of generality, we assume that $\gamma_{\mathtt{k}}= \max\{\gamma_1,\gamma_2,\cdots,\gamma_{\mathtt{k}}\}$. A straightforward calculation gives
$$ \gamma_{\mathtt{k}}= \frac{1+ p_{\mathtt{k}}+ p_{\mathtt{k}-1}p_{\mathtt{k}}+\cdots+ p_2p_3\cdots p_{\mathtt{k}}}{\prod^{\mathtt{k}}_{\ell=1} p_\ell-1}. $$
From the condition \eqref{Critical_Condition}, one gains
\begin{equation}
\gamma_{\mathtt{k}}\ge \frac{d}{2}, \quad\text{ that is, }\quad \Gamma(n,p_1,p_2,\cdots,p_{\mathtt{k}}):= \gamma_{\mathtt{k}}- \frac{d}{2}\ge 0. \label{Critical_Condition*}
\end{equation}
Let us now separate our consideration into two cases as follows:
\begin{itemize}[leftmargin=*]
\item[$\bullet$] \textbf{Case 1:} $d\ge3$ and $\beta\ne 0$. At first, let us devote our attention to the subcritical case $\Gamma(n,p_1,p_2,\cdots,p_{\mathtt{k}})>0$. Now we plug the above chain of estimates into (\ref{proof.3}) successively to achieve
\begin{align}
\mathcal{I}_R[u_{\mathtt{k}-1}]+ C^0_\mathtt{k} \varepsilon &\lesssim R^{d-\frac{d+2}{p_1}}\big(\mathcal{I}_R[u_{\mathtt{k}}]\big)^{\frac{1}{p_1}} \qquad \big(\text{since }\phi_R^*(t,x)\le \phi_R(t,x) \text{ in }Q_R\big) \nonumber \\
&\lesssim R^{d-\frac{d+2}{p_1}}\left(R^{d-\frac{d+2}{p_2}}\left(\int_{Q_R}|u_1(t,x)|^{p_2}\,\Psi(x)\,\phi_R^*(t,x)d(t,x)\right)^{\frac{1}{p_2}}\right)^{\frac{1}{p_1}} \nonumber \\
&\quad= R^{d-\frac{2}{p_1}-\frac{d+2}{p_1p_2}}\left(\int_{Q_R}|u_1(t,x)|^{p_2}\,\Psi(x)\,\phi_R^*(t,x)d(t,x)\right)^{\frac{1}{p_1p_2}} \nonumber \\
&\lesssim R^{d-\frac{2}{p_1}-\frac{d+2}{p_1p_2}}\big(\mathcal{I}_R[u_1]\big)^{\frac{1}{p_1p_2}} \nonumber \\
&\quad \vdots \nonumber \\
&\lesssim R^{d- 2\big(\frac{1}{p_1}+ \frac{1}{p_1p_2}+\cdots+\frac{1}{p_1p_2\cdots p_{\mathtt{k}-1}}\big)- \frac{d+2}{p_1p_2\cdots p_{\mathtt{k}}}}\mathcal{I}_R[u_{\mathtt{k}-1}]^{\frac{1}{p_1p_2\cdots p_{\mathtt{k}}}} \nonumber \\
&\quad= R^{- 2\big(\frac{1}{p_1}+ \frac{1}{p_1p_2}+\cdots+\frac{1}{p_1p_2\cdots p_{\mathtt{k}}}\big)+ d\big(1- \frac{1}{p_1p_2\cdots p_{\mathtt{k}}}\big)}\mathcal{I}_R[u_{\mathtt{k}-1}]^{\frac{1}{p_1p_2\cdots p_{\mathtt{k}}}}, \label{proof.4}
\end{align}
which is equivalent to
\begin{equation}
C^0_\mathtt{k} \varepsilon \lesssim R^{- 2\big(\frac{1}{p_1}+ \frac{1}{p_1p_2}+\cdots+\frac{1}{p_1p_2\cdots p_{\mathtt{k}}}\big)+ d\big(1- \frac{1}{p_1p_2\cdots p_{\mathtt{k}}}\big)}\mathcal{I}_R[u_{\mathtt{k}-1}]^{\frac{1}{p_1p_2\cdots p_{\mathtt{k}}}}- \mathcal{I}_R[u_{\mathtt{k}-1}]. \label{proof.5}
\end{equation}
Thus, the employment of the elementary inequality
$$ A\,y^s- y \le A^{\frac{1}{1-s}} \quad \text{ for any } A>0,\, y \ge 0 \text{ and } 0< s< 1 $$
follows immediately
\begin{align}
&R^{- 2\big(\frac{1}{p_1}+ \frac{1}{p_1p_2}+\cdots+\frac{1}{p_1p_2\cdots p_{\mathtt{k}}}\big)+ d\big(1- \frac{1}{p_1p_2\cdots p_{\mathtt{k}}}\big)}\mathcal{I}_R[u_{\mathtt{k}-1}]^{\frac{1}{p_1p_2\cdots p_{\mathtt{k}}}}- \mathcal{I}_R[u_{\mathtt{k}-1}] \nonumber \\
&\hspace{5cm} \le R^{-2\cdot\frac{1+p_{\mathtt{k}}+\cdots+ p_2p_3\cdots p_{\mathtt{k}}}{p_1p_2\cdots p_{\mathtt{k}}-1}+ d} \label{proof.6}
\end{align}
for all $R > R_0$. Linking (\ref{proof.5}) and (\ref{proof.6}) we deduce
\begin{equation}
C^0_\mathtt{k} \varepsilon \lesssim R^{-2\cdot\frac{1+p_{\mathtt{k}}+\cdots+ p_2p_3\cdots p_{\mathtt{k}}}{p_1p_2\cdots p_{\mathtt{k}}-1}+ d} \label{proof.7}
\end{equation}
for all $R > R_0$. Because of the assumption $\Gamma(n,p_1,p_2,\cdots,p_{\mathtt{k}})>0$, letting $R\to \sqrt{T_\varepsilon}$ in (\ref{proof.7}) we obtain
$$ T_\varepsilon\le C\varepsilon^{-\Gamma(n,p_1,p_2,\cdots,p_{\mathtt{k}})^{-1}}. $$
This is the third estimate for lifespan of solutions what we wanted to prove. The next step is to focus on the critical case $\Gamma(n,p_1,p_2,\cdots,p_{\mathtt{k}})=0$, where there exist two exponents $p_{j_1}\ne p_{j_2}$ with $j_1,j_2\in \{1,2,\cdots,\mathtt{k}\}$ and $j_1\ne j_2$. After repeating an analogous manner to (\ref{proof.4}), one may arrive at the following estimate:
\begin{align}
\mathcal{I}_R[u_{\mathtt{k}-1}]+ C^0_\mathtt{k} \varepsilon &\lesssim R^{- 2\big(\frac{1}{p_1}+ \frac{1}{p_1p_2}+\cdots+\frac{1}{p_1p_2\cdots p_{\mathtt{k}}}\big)+ d\big(1- \frac{1}{p_1p_2\cdots p_{\mathtt{k}}}\big)} \nonumber \\
&\qquad \times \left(\int_{Q_R}|u_{\mathtt{k}-1}(t,x)|^{p_\mathtt{k}}\,\Psi(x)\,\phi_R^*(t,x)d(t,x)\right)^{\frac{1}{p_1p_2\cdots p_{\mathtt{k}}}} \nonumber \\
&\quad= \left(\int_{Q_R}|u_{\mathtt{k}-1}(t,x)|^{p_\mathtt{k}}\,\Psi(x)\,\phi_R^*(t,x)d(t,x)\right)^{\frac{1}{p_1p_2\cdots p_{\mathtt{k}}}}, \label{proof.8}
\end{align}
where we note that the assumption $\Gamma(n,p_1,p_2,\cdots,p_{\mathtt{k}})=0$ is used. Let us now define the following auxiliary functionals:
\begin{align*}
h_{p_\mathtt{k}}= h_{p_\mathtt{k}}(r) &= \int_{Q_R}|u_{\mathtt{k}-1}(t,x)|^{p_\mathtt{k}}\,\Psi(x)\,\phi_r^*(t,x)d(t,x)\ \ \mbox{and}\ \ \mathcal{H}_{p_\mathtt{k}}= \mathcal{H}_{p_\mathtt{k}}(R)=\int_0^R h_{p_\mathtt{k}}(r)r^{-1}dr.
\end{align*}
On the one hand, carrying out the change of variable $\rho=\f{t^2+(|x|-1)^4}{r^4}$ one derives
\begin{align}
\mathcal{H}_{p_\mathtt{k}}(R)&=\int_0^R \left(\int_0^{T_{\varepsilon}}\int_{\Omega}|u_{\mathtt{k}-1}(t,x)|^{p_\mathtt{k}}\,\Psi(x)\,\phi_r^*(t,x)d(t,x)\right)\,r^{-1}dr \nonumber \\
&=\frac{1}{4} \int_{Q_R}|u_{\mathtt{k}-1}(t,x)|^{p_\mathtt{k}}\,\Psi(x)\,\Big(\int_{\frac{t^2+(|x|-1)^4}{R^4}}^{\infty}\big(\varphi^*(\rho)\big)^{\lambda+2}\rho^{-1}d\rho\Big) d(t,x) \nonumber \\
&\le \frac{1}{4} \int_{Q_R}|u_{\mathtt{k}-1}(t,x)|^{p_\mathtt{k}}\,\Psi(x)\,\left(\int_{1/2}^1\big(\varphi^*(\rho)\big)^{\lambda+2}\rho^{-1}d\rho\right)d(t,x), \label{proof.9}
\end{align}
where the support condition for $\varphi^*(\rho)$ is applied to \eqref{proof.9}. On the other hand, using the property $\varphi^*(\rho)\equiv \varphi(\rho)$ for any $\rho\in[1/2,1]$ in \eqref{proof.9} we have
\begin{align}
\mathcal{H}_{p_\mathtt{k}}(R) &\le \frac{1}{4}\int_{Q_R}|u_{\mathtt{k}-1}(t,x)|^{p_\mathtt{k}}\,\Psi(x)\,\sup_{r \in (0,R)}\left(\varphi\left(\frac{t^2+(|x|-1)^4}{r^4}\right)\right)^{\lambda+2} \left(\int_{1/2}^1 \rho^{-1}d\rho\right)d(t,x) \nonumber \\
&\le \frac{\log 2}{4} \int_{Q_R}|u_{\mathtt{k}-1}(t,x)|^{p_\mathtt{k}}\,\Psi(x)\,\left(\varphi\left(\frac{t^2+(|x|-1)^4}{R^4}\right)\right)^{\lambda+2}d(t,x) \nonumber \\
&\quad= \frac{\log 2}{4} \int_{Q_R}|u_{\mathtt{k}-1}(t,x)|^{p_\mathtt{k}}\,\Psi(x)\,\phi_R(t,x)d(t,x)d(t,x)= \frac{\log 2}{4}\mathcal{I}_R[u_{\mathtt{k}-1}]. \label{proof.10}
\end{align}
Moreover, it is obvious to recognize the relation
\begin{equation}
h_{p_\mathtt{k}}(R)= R\mathcal{H}'_{p_\mathtt{k}}(R). \label{proof.11}
\end{equation}
Hence, collecting all the obtained estimates from \eqref{proof.8} to \eqref{proof.11} we conclude
$$ \frac{4}{\log 2}\mathcal{H}_{p_\mathtt{k}}(R)+ C^0_\mathtt{k} \varepsilon \lesssim \left(R\mathcal{H}'_{p_\mathtt{k}}(R)\right)^{\frac{1}{p_1p_2\cdots p_{\mathtt{k}}}}, $$
that is,
$$ R^{-1}\le C\left(\frac{4}{\log 2}\mathcal{H}_{p_\mathtt{k}}(R)+ C^0_\mathtt{k} \varepsilon\right)^{-p_1p_2\cdots p_{\mathtt{k}}}\mathcal{H}'_{p_\mathtt{k}}(R). $$
Taking integration of both the sides of the above estimate over $[R_0,\sqrt{T_\varepsilon}]$ one finds
\begin{align*}
\log\sqrt{T_\varepsilon}- \log R_0 &\le C\int_{R_0}^{T_\varepsilon}\left(\frac{4}{\log 2}\mathcal{H}_{p_\mathtt{k}}(R)+ C^0_\mathtt{k} \varepsilon\right)^{-p_1p_2\cdots p_{\mathtt{k}}}\mathcal{H}'_{p_\mathtt{k}}(R)dR \\ 
&\le \frac{C\log 2}{4(1- p_1p_2\cdots p_{\mathtt{k}})}\left(\frac{4}{\log 2}\mathcal{H}_{p_\mathtt{k}}(R)+ C^0_\mathtt{k} \varepsilon\right)^{1-p_1p_2\cdots p_{\mathtt{k}}}\Big|_{R=R_0}^{R=\sqrt{T_\varepsilon}} \\
&\le \frac{C\log 2}{4(p_1p_2\cdots p_{\mathtt{k}}-1)}\left(\frac{4}{\log 2}\mathcal{H}_{p_\mathtt{k}}(R_0)+ C^0_\mathtt{k} \varepsilon\right)^{1-p_1p_2\cdots p_{\mathtt{k}}} \\
&\le \frac{C\,(C^0_\mathtt{k})^{1-p_1p_2\cdots p_{\mathtt{k}}}\log 2}{4(p_1p_2\cdots p_{\mathtt{k}}-1)}\varepsilon^{1-p_1p_2\cdots p_{\mathtt{k}}},
\end{align*}
which yields
$$ T_\varepsilon \le \exp\left(\frac{C\,(C^0_\mathtt{k})^{1-p_1p_2\cdots p_{\mathtt{k}}}\log 2}{2(p_1p_2\cdots p_{\mathtt{k}}-1)}\varepsilon^{-(p_1p_2\cdots p_{\mathtt{k}}-1)}+ 2\log R_0\right). $$
This is to indicate the next desired estimate for lifespan of solutions. Finally, we will give our verification to the critical case $\Gamma(n,p_1,p_2,\cdots,p_{\mathtt{k}})=0$ when $p_1=p_2=\cdots=p_{\mathtt{k}}$. It follows that
$$ p_1=p_2=\cdots=p_{\mathtt{k}}= 1+\frac{2}{n}=:p^*. $$
Obviously, one should realize that the following relation holds:
\begin{align}
&\partial^2_t (u_1+u_2+\cdots+u_{\mathtt{k}})- \Delta (u_1+u_2+\cdots+u_{\mathtt{k}})+ \partial_t (u_1+u_2+\cdots+u_{\mathtt{k}}) \nonumber \\
&\qquad = |u_1|^{p^*}+ |u_2|^{p^*}+\cdots+|u_{\mathtt{k}}|^{p^*}\ge C|u_1+u_2+\cdots+u_{\mathtt{k}}|^{p^*}. \label{proof.12}
\end{align}
For this reason, we consider the treatment of the system \eqref{Equation_Main} as that of the single equations \eqref{Equation_2} and \eqref{Equation_3} accompanied with the zero Robin boundary condition. Then, following the same approach as in the paper \cite{IkedaSobajima2019-1} we may arrive at
$$ T_\varepsilon \le \exp\left(C\varepsilon^{-(p^*-1)}\right). $$

\item[$\bullet$] \textbf{Case 2:} $d=2$ and $\beta\ne 0$. Repeating an argument as we did in Case $1$ we give the following estimates:
\begin{align}
\mathcal{I}_R[u_{\mathtt{k}-1}]+ C^0_\mathtt{k} \varepsilon &\lesssim R^{2-\frac{4}{p_1}}(\log R)^{1-\frac{1}{p_1}}\big(\mathcal{I}_R[u_{\mathtt{k}}]\big)^{\frac{1}{p_1}} \qquad \big(\text{since }\phi_R^*(t,x)\le \phi_R(t,x) \text{ in }Q_R\big) \nonumber \\
&\lesssim R^{2-\frac{4}{p_1}}(\log R)^{1-\frac{1}{p_1}}\left(R^{2-\frac{4}{p_2}}(\log R)^{1-\frac{1}{p_2}}\left(\int_{Q_R}|u_1(t,x)|^{p_2}\,\Psi(x)\,\phi_R^*(t,x)d(t,x)\right)^{\frac{1}{p_2}}\right)^{\frac{1}{p_1}} \nonumber \\
&\quad = R^{2-\frac{2}{p_1}-\frac{4}{p_1p_2}}(\log R)^{1-\frac{1}{p_1p_2}}\left(\int_{Q_R}|u_1(t,x)|^{p_2}\,\Psi(x)\,\phi_R^*(t,x)d(t,x)\right)^{\frac{1}{p_1p_2}} \nonumber \\
&\lesssim R^{2-\frac{2}{p_1}-\frac{4}{p_1p_2}}(\log R)^{1-\frac{1}{p_1p_2}}\big(\mathcal{I}_R[u_1]\big)^{\frac{1}{p_1p_2}} \nonumber \\
&\quad \vdots \nonumber \\
&\lesssim R^{2- 2\big(\frac{1}{p_1}+ \frac{1}{p_1p_2}+\cdots+\frac{1}{p_1p_2\cdots p_{\mathtt{k}-1}}\big)- \frac{4}{p_1p_2\cdots p_{\mathtt{k}}}}(\log R)^{1-\frac{1}{p_1p_2\cdots p_{\mathtt{k}}}} \nonumber \\
&\hspace{3cm} \times\left(\int_{Q_R}|u_{\mathtt{k}-1}(t,x)|^{p_\mathtt{k}}\,\Psi(x)\,\phi_R^*(t,x)d(t,x)\right)^{\frac{1}{p_1p_2\cdots p_{\mathtt{k}}}} \nonumber \\
&\quad= R^{- 2\big(\frac{1}{p_1}+ \frac{1}{p_1p_2}+\cdots+\frac{1}{p_1p_2\cdots p_{\mathtt{k}}}\big)+ 2\big(1- \frac{1}{p_1p_2\cdots p_{\mathtt{k}}}\big)}(\log R)^{1-\frac{1}{p_1p_2\cdots p_{\mathtt{k}}}} \nonumber \\
&\hspace{3cm} \times\left(\int_{Q_R}|u_{\mathtt{k}-1}(t,x)|^{p_\mathtt{k}}\,\Psi(x)\,\phi_R^*(t,x)d(t,x)\right)^{\frac{1}{p_1p_2\cdots p_{\mathtt{k}}}}. \label{proof.13}
\end{align}
We pay attention that the subcritical case is equivalent to $\Gamma(2,p_1,p_2,\cdots,p_{\mathtt{k}})>0$. Then, the immediate employment of Lemma \ref{Useful_lemma_2D} to (\ref{proof.13}) with the function $\eta(t,x)=|u_{\mathtt{k}-1}(t,x)|^{p_\mathtt{k}}\,\Psi(x)$ and the parameters
$$\omega= C^0_\mathtt{k} \varepsilon,\,\sigma= 2\Gamma(2,p_1,p_2,\cdots,p_{\mathtt{k}}),\,\mu=1 \text{ and }p= p_1p_2\cdots p_{\mathtt{k}}$$
leads to
$$\sqrt{T_\ep}\le C\Big(\ep^{-1}\log\big(\ep^{-1}\big)\Big)^{\frac{1}{2\Gamma(2,p_1,p_2,\cdots,p_{\mathtt{k}})}}, $$
which implies what we wanted to show. When the critical case occurs, i.e. $\Gamma(2,p_1,p_2,\cdots,p_{\mathtt{k}})=0$, with $p_1=p_2=\cdots=p_{\mathtt{k}}$, we have $p_1=p_2=\cdots=p_{\mathtt{k}}=2$. Therefore, noticing again the relation \eqref{proof.12} we may apply an analogous strategy used in the paper \cite{IkedaSobajima2019-1} to conclude
$$ T_\varepsilon \le \exp\exp\Big(C\ep^{-1}\Big). $$

\item[$\bullet$] \textbf{Case 3:} $d\ge 2$ and $\beta= 0$. By carrying out the same procedure as we did in Case $1$, one may conclude that all the desired estimates for lifespan in Theorem \ref{Upper_Bound_Thoerem} hold.
\end{itemize}

\noindent Summarizing, our proof is completed.

\section{Concluding remarks}\label{Sec.Conlusions}
\begin{remark}
In this paper, we have employed the test function method with a special choice of test functions, which approximates the harmonic functions fulfilling some of typical boundary conditions (the Dirichlet/Neumann/Robin boundary condition) on $\partial \Omega$ for any $d\ge 2$ separately, to indicate both the blow-up result in a finite time and the sharp upper bound of lifespan estimates for small solutions to \eqref{Equation_Main}. This remark tells us that the used method can be also directly applicable to the weakly coupled system of semi-linear classical wave equations with double damping terms in an exterior domain as follows:
\begin{equation}\label{Equation_Double.Damping}
\begin{cases}
\partial^2_t u_1(t,x)- \Delta u_1(t,x)+ \partial_t u_1(t,x)- \Delta \partial_t u_1(t,x)= |u_{\mathtt{k}}(t,x)|^{p_1}, &(t,x) \in (0,T) \times \Omega, \\
\partial^2_t u_2(t,x)- \Delta u_2(t,x)+ \partial_t u_2(t,x)- \Delta \partial_t u_2(t,x)= |u_1(t,x)|^{p_2}, &(t,x) \in (0,T) \times \Omega, \\
\quad \vdots \\
\partial^2_t u_{\mathtt{k}}(t,x)- \Delta u_{\mathtt{k}}(t,x)+ \partial_t u_{\mathtt{k}}(t,x)- \Delta \partial_t u_{\mathtt{k}}(t,x)= |u_{\mathtt{k}-1}(t,x)|^{p_{\mathtt{k}}}, &(t,x) \in (0,T) \times \Omega, \\
\alpha\f{\partial u_\ell}{\partial n^+}(t,x)+\beta u_\ell(t,x)=0, &(t,x) \in (0,T) \times \partial\Omega, \\
u_\ell(0,x)= \varepsilon u_{0,\ell}(x),\quad \partial_t u_\ell(0,x)= \varepsilon u_{1,\ell}(x), &x \in \Omega,
\end{cases}
\end{equation}
with $\ell=1,2,\cdots,\mathtt{k}$ and for any $d\ge 2$. Following an analogous approach to this paper we may claim that the statements of Theorem \ref{LWP_Prop} and Theorem \ref{Upper_Bound_Thoerem} still remain valid to \eqref{Equation_Double.Damping}. Additionally, in term of studying the special case of \eqref{Equation_Double.Damping} when $\mathtt{k}=1$ and $\alpha=0$ we refer the interested readers to the recent paper of D'Abbicco-Ikehata-Takeda \cite{DAbbiccoIkehataTakeda2019}. However, their paper did not report any information about lifespan estimates for solutions, especially, a blow-up result for the critical case $p=2$ with $d=2$ was also not included. For this reason, we can say that our results are not only to fill this lack of the cited paper in the case of the Dirichlet boundary condition but also to extend the further results for the Neumann and Robin boundary conditions.
\end{remark}

\begin{remark}
Speaking about a desired result for estimating upper bound of lifespan of solutions to \eqref{Equation_Main} in one-dimensional case, let's take the harmonic function $\Psi= \Psi(x)$ given by
$$\Psi(x)= \begin{cases}
|x|-1+\f{\alpha}{\beta} &\text{ if }\ \ \beta\ne 0,\\
1 &\text{ if }\ \ \beta=0,
\end{cases}$$
in place of (\ref{Test_function_x}) and (\ref{Test_function_x*}), respectively. By carrying out several straightforward calculations as we did in Lemma \ref{Useful_Lemma}, one sees that all the auxiliary estimates in Lemma \ref{Useful_Lemma} are also true for $d=1$. Then, with the same assumptions as in Theorem \ref{Upper_Bound_Thoerem}, we may repeat some similar steps to those in the proof of Theorem \ref{Upper_Bound_Thoerem} for Case $1$ to conclude the following estimates for $d=1$:
\begin{equation}\label{Lifespan_1D}
{\rm LifeSpan}(u_1,u_2,\cdots,u_{\mathtt{k}})\le \begin{cases}
C\varepsilon^{-(\max\{\gamma_1,\gamma_2,\cdots,\gamma_{\mathtt{k}}\}- 1)^{-1}} &\text{ if }\ \ \max\{\gamma_1,\gamma_2,\cdots,\gamma_{\mathtt{k}}\}> 1,\ \ \beta\ne 0,\\
C\varepsilon^{-(\max\{\gamma_1,\gamma_2,\cdots,\gamma_{\mathtt{k}}\}- 1/2)^{-1}} &\text{ if }\ \ \max\{\gamma_1,\gamma_2,\cdots,\gamma_{\mathtt{k}}\}> 1/2,\ \ \beta= 0.
\end{cases}
\end{equation}
To establish this, we determine the function $\Theta_{p_\ell}=\Theta_{p_\ell}(R)$ by
$$\Theta_{p_\ell}(R)= \begin{cases}
R^{2-\frac{4}{p_\ell}} &\text{ if }\ \ \beta\ne 0,\\
R^{1-\frac{3}{p_\ell}} &\text{ if }\ \ \beta=0,
\end{cases}$$
with $\ell=1,2,\cdots,\mathtt{k}$ instead of (\ref{Auxiliary_Funct}), which comes from the following observations:
$$\int_{Q_R^*}\Psi(x)d(t,x)\lesssim \begin{cases}
R^2(R+1)^2\approx R^4 &\text{ if }\ \ \beta\ne 0,\\
R^2(R+1)\approx R^3 &\text{ if }\ \ \beta=0,
\end{cases}$$
when $d=1$. One recognizes that if we replace $\mathtt{k}=1$, then the relation $\max\{\gamma_1,\gamma_2,\cdots,\gamma_{\mathtt{k}}\}> 1$ is equivalent to $p_1<2$ for $d=1$. Obviously, there is a gap between $2$ and $p_{\rm Fuj}(1)=3$ so that the obtained upper bound in (\ref{Lifespan_1D}) seems to be not sharp when the Dirichlet boundary condition is considered, i.e. $\alpha= 0$ and $\beta\ne 0$. However, this phenomenon is reasonable due to the fact no manuscript has succeeded in indicating what a critical exponent is in this case so far (see also \cite{Ikehata2003}).
\end{remark}

\section*{Appendix}
\begin{lemma}\label{Useful_lemma_2D}
Let $\omega>0$, $C_0>0$, $R_1>0$, $\sigma\ge 0$ and $\mu>0$. We assume $0\le \eta= \eta(t,x)\in L^1_{\rm loc}\Big([0,T),L^1(\Omega)\Big)$ for $T>R_1$ satisfying the following inequality:
$$\omega+ \int_{Q_R}\eta(t,x) \phi_R(t,x)\,d(t,x)\le C_0\,R^{-\frac{\sigma}{p'}}\big(\log R\big)^{\frac{\mu}{p'}}\left(\int_{Q_R}\eta(t,x) \phi^*_R(t,x)\,d(t,x)\right)^{\frac{1}{p}} $$
for any $R\in \Big[R_1,\sqrt{T}\Big)$, where $\phi_R(t,x)$, $\phi^*_R(t,x)$ and $Q_R$ are introduced as in Section \ref{Test_functions}. Then, they hold
$$\sqrt{T}\le 
\begin{cases}
C\omega^{-\frac{1}{\sigma}}\Big(\log\big(\omega^{-1}\big)\Big)^{\frac{\mu}{\sigma}} &\text{ if }\ \ \sigma>0 \text{ and } \mu>0, \\
\exp\Big(C\omega^{-\frac{p-1}{1-\mu(p-1)}}\Big) &\text{ if }\ \ \sigma=0 \text{ and } 0<\mu<\frac{1}{p-1}, \\
\exp\exp\Big(C\omega^{-(p-1)}\Big) &\text{ if }\ \ \sigma=0 \text{ and } \mu=\frac{1}{p-1},
\end{cases}$$
where $C$ is a positive constant independent of $\omega$.
\end{lemma}
\begin{proof}
We follow the proof of Lemma 2.2 in \cite{IkedaSobajima2019-1} with a minor modification to conclude the desired estimates above.
\end{proof}

\section*{Acknowledgements}
Tuan Anh Dao was funded by Vingroup JSC and supported by the Postdoctoral Scholarship Programme of Vingroup Innovation Foundation (VINIF), Institute of Big Data, code VINIF.2021.STS.16. Masahiro Ikeda is supported by JST CREST Grant Number JPMJCR1913, Japan and Grant-in-Aid for Young Scientists Research (No.19K14581), Japan Society for the Promotion of Science.


\begin{thebibliography}{99}

\bibitem{Adams75} R. A. Adams, \textit{Sobolev spaces}, Academic Press, New York, 1975.

\bibitem{ArendtUrban10} W. Arendt and K. Urban, \textit{Partielle Differenzialgleichungen}, Spektrum, Heidelberg, 2010.

\bibitem{BarasPierre1985} P. Baras, M. Pierre, \textit{Crit\`{e}re d'existence de solutions positives pour des \'{e}quations semi-lin\'{e}aires non monotones}, Ann. Inst. H. Poincar\'{e} Anal. Non Lin\'{e}aire, \textbf{2} (1985), 185--212.

\bibitem{ChenDao2021} W. Chen, T.A. Dao, \textit{Sharp lifespan estimates for the weakly coupled system of semilinear damped wave equations in the critical case}, Math. Ann., (2022), to appear.

\bibitem{CaHa98} T. Cazenave, A. Haraux, \textit{An Introduction to Semilinear Evolution Equations}, Oxford University Press, 1998.

\bibitem{DAbbiccoIkehataTakeda2019} M. D'Abbicco, R. Ikehata, H. Takeda, \textit{Critical exponent for semi-linear wave equations with double damping terms in exterior domains}, Nonlinear Differ. Equ. Appl., \textbf{26} (2019), 56.

\bibitem{DaoReissig2021} T.A. Dao, M. Reissig, \textit{The interplay of critical regularity of nonlinearities in a weakly coupled system of semi-linear damped wave equations}, J. Differential Equations, \textbf{299} (2021), 1--32.

\bibitem{Fujita1966} H. Fujita, \textit{On the blowing up of solutions of the Cauchy problem for $u_t= \Delta u+ |u|^{1+\alpha}$}, J. Sci. Univ. Tokyo Sec. I, \textbf{13} (1966), 109--124.

\bibitem{FinoIbrahimWehbe2017} A.Z. Fino, H. Ibrahim, A. Wehbe, \textit{A blow-up result for a nonlinear damped wave equation in exterior domain: The critical case}, Comput. Math. Appl., \textbf{73} (2017), 2415--2420.

\bibitem{FujiwaraIkedaWakasugi2019} K. Fujiwara, M. Ikeda, Y. Wakasugi, \textit{Estimates of lifespan and blow-up rates for the wave equation with a time-dependent damping and a power-type nonlinearity}, Funkc. Ekvac., \textbf{62} (2019), 157--189.

\bibitem{Hayakawa1973} K. Hayakawa, \textit{On nonexistence of global solutions of some semilinear parabolic differential equations}, Proc. Japan Acad., \textbf{49} (1973), 503--505.

\bibitem{HayashiNaumkinTominaga2015} N. Hayashi, P. I. Naumkin, M. Tominaga, \textit{Remark on a weakly coupled system of nonlinear damped wave equations}, J. Math. Anal. Appl., \textbf{428} (2015), 490--501.

\bibitem{Ikehata2002} R. Ikehata, \textit{Small data global existence of solutions for dissipative wave equations in an exterior domain}, Funkcial. Ekvac., \textbf{44} (2002), 259--269.

\bibitem{Ikehata2003} R. Ikehata, \textit{A remark on a critical exponent for the semilinear dissipative wave equation in the one dimensional half space}, Differential Integral Equations, \textbf{16} (2003), 727--736.

\bibitem{Ikehata2004} R. Ikehata, \textit{Global existence of solutions for semilinear damped wave equation in 2-D exterior domain}, J. Differential Equations, \textbf{200} (2004), 53--68.

\bibitem{Ikehata2005} R. Ikehata, \textit{Two dimensional exterior mixed problem for semilinear damped wave equations}, J. Math. Anal. Appl., \textbf{301} (2005), 366--377.

\bibitem{IkedaOgawa2016} M. Ikeda, T. Ogawa, \textit{Lifespan of solutions to the damped wave equation with a critical nonlinearity}, J. Differential Equations, \textbf{261} (2016), 1880--1903.

\bibitem{IkedaSobajima2019-1} M. Ikeda, M. Sobajima, \textit{Remark on upper bound for lifespan of solutions to semilinear evolution equations in a two-dimensional exterior domain}, J. Math. Anal. Appl., \textbf{470} (2019), 318--326.

\bibitem{IkedaSobajima2019-2} M. Ikeda, M. Sobajima, \textit{Sharp upper bound for lifespan of solutions to some critical semilinear parabolic, dispersive and hyperbolic equations via a test function method}, Nonlinear Anal., \textbf{182} (2019), 57--74.

\bibitem{IkedaTaniguchiWakasugiArxiv} M. Ikeda, K. Taniguchi, Y. Wakasugi, \textit{Global existence and asymptotic behavior for nonlinear damped wave equations on measure spaces}, arxiv:2106.10322.

\bibitem{IkedaWakasugi2020} M. Ikeda, Y. Wakasugi, \textit{A note on the lifespan of solutions to the semilinear damped wave equation}, Proc. Amer. Math. Soc., \textbf{148} (2015), 157--172.

\bibitem{IkedaJleliSamet2020} M. Ikeda, M. Jleli, B. Samet, \textit{On the existence and nonexistence of global solutions
for certain semilinear exterior problems with nontrivial
Robin boundary conditions}, J. Differential Equations, \textbf{269} (2020), 563--594.

\bibitem{JleliSamet2019} M. Jleli, B. Samet, \textit{New blow-up results for nonlinear boundary value problems in exterior domains}, Nonlinear Anal., \textbf{178} (2019), 348--365.

\bibitem{KiraneQafsaoui2002} M. Kirane, M. Qafsaoui, \textit{Fujita's exponent for a semilinear wave equation with linear damping}, Adv. Nonlinear Stud., \textbf{2} (2002), 41--49.

\bibitem{LaiYin2017} N. Lai, S. Yin, \textit{Finite time blow-up for a kind of initial-boundary value problem of semilinear damped wave equation}, Math. Methods Appl. Sci., \textbf{40} (2017), 1223--1230.

\bibitem{LaiZhou2019} N.A. Lai, Y. Zhou, \textit{The sharp lifespan estimate for semilinear damped wave equation with Fujita critical power in higher dimensions}, J. Math. Pures Appl. (9), \textbf{123} (2019), 229--243.

\bibitem{LiZhou1995} T. T. Li, Y. Zhou, \textit{Breakdown of solutions to $\Box u+u_t=|u|^{1+\alpha}$}, Discrete Contin. Dyn. Syst., \textbf{1} (1995), 503--520.

\bibitem{MitidieriPohozaev2001} E. Mitidieri, S.I. Pohozaev, \textit{A priori estimates and blow-up of solutions to nonlinear partial differential equations and inequalities}, Proc. Steklov. Inst. Math., \textbf{234} (2001), 1--383.

\bibitem{Narazaki2009} T. Narazaki, \textit{Global solutions to the Cauchy problem for the weakly coupled system of damped wave equations}, AIMS Proceedings Discrete Contin. Dyn. Syst., (2009), 592--601.

\bibitem{Narazaki2011} T. Narazaki, \textit{Global solutions to the Cauchy problem for a system of damped wave equations}, Differential Integral Equations, \textbf{24} (2011), 569--600.

\bibitem{Nishihara2003} K. Nishihara, \textit{$L^p-L^q$ estimates for the $3$-D damped wave equation and their application to the semilinear problem}, Sem. Notes Math. Sci., vol \textbf{6}. Ibaraki Univ. (2003), 69--83.

\bibitem{Nishihara2011} K. Nishihara, \textit{Asymptotic behavior of solutions to the semilinear wave equation with time-dependent damping}, Tokyo J. Math., \textbf{34} (2011), 327--343.

\bibitem{NishiharaWakasugi2014} K. Nishihara, Y. Wakasugi, \textit{Critical exponent for the Cauchy problem to the weakly coupled damped wave system}, Nonlinear Anal., \textbf{108} (2014), 249--259.

\bibitem{NishiharaWakasugi2015} K. Nishihara, Y. Wakasugi, \textit{Global existence of solutions for a weakly coupled system of semilinear damped wave equations}, J. Differential Equations, \textbf{259} (2015), 4172--4201.

\bibitem{Ono2003} K. Ono, \textit{Decay estimates for dissipative wave equations in exterior domains}, J. Math. Anal. Appl., \textbf{286} (2003), 540--562.

\bibitem{OgawaTakeda2009} T. Ogawa, H. Takeda, \textit{Non-existence of weak solutions to nonlinear damped wave equations in exterior domains}, Nonlinear Anal., \textbf{70} (2009), 3696--3701.

\bibitem{SunWang2007} F. Sun, M. Wang, \textit{Existence and nonexistence of global solutions for a nonlinear hyperbolic system with damping}, Nonlinear Anal., \textbf{66} (2007), 2889--2910.

\bibitem{Sobajima2019-1} M. Sobajima, \textit{Remarks on test function methods for blowup of solutions to semilinear evolution equations in sectorial domain}, RIMS, Kokyuroku, \textbf{2121} (2019), 63--73.

\bibitem{Sobajima2019-2} M. Sobajima, \textit{Global existence of solutions to semilinear damped wave equation with slowly decaying initial data in exterior domain}, Differential Integral Equations, \textbf{32} (2019), 615--638.

\bibitem{Sobajima2021} M. Sobajima, \textit{Global existence of solutions to a weakly coupled critical parabolic system in two-dimensional exterior domains}, J. Math. Anal. Appl., \textbf{501} (2021), 125214.

\bibitem{TodorovaYordanov2001} G. Todorova, B. Yordanov. \textit{Critical exponent for a nonlinear wave equation with damping}, J. Differential Equations, \textbf{174} (2001), 464--489.

\bibitem{Takeda2009} H. Takeda, \textit{Global existence and nonexistence of solutions for a system of nonlinear damped wave equations}, J. Math. Anal. Appl., \textbf{360} (2009), 631--650.

\bibitem{Weissler1981} F.B. Weissler, \textit{Existence and nonexistence of global solutions for a semilinear heat equation}, Israel J. Math., \textbf{38} (1981), 29--40.

\bibitem{Zhang2001} Q.S. Zhang, \textit{A blow-up result for a nonlinear wave equation with damping: the critical case}, C. R. Acad. Sci. Paris S\'{e}r. I Math., \textbf{333} (2001), 109--114.

\end{thebibliography}
\end{document}